\documentclass[11pt]{amsart}
\usepackage[utf8]{inputenc}
\usepackage[left=2cm, right=2cm, top=2cm, bottom=2cm]{geometry}
\usepackage{mdframed, mathtools, esint, amsmath, enumitem, amscd, amssymb, indentfirst, latexsym, multicol, verbatim, mdframed, amsthm, esint}
\usepackage[ddmmyyyy]{datetime}
\usepackage[dvipsnames,table,xcdraw]{xcolor}
\usepackage[T1]{fontenc}
\usepackage[sort, numbers]{natbib}
\usepackage{ulem}

\newcommand{\mres}{\mathbin{\vrule height 1.6ex depth 0pt width
0.13ex\vrule height 0.13ex depth 0pt width 1.3ex}}
\newcommand{\overbar}[1]{\mkern 1.5mu\overline{\mkern-1.5mu#1\mkern-1.5mu}\mkern 1.5mu}

\newcommand{\R}{\mathbb{R}}
\newcommand*\dd{\mathop{}\!\mathrm{d}}
\renewcommand{\div}{\text{div}} 
\renewcommand{\phi}{\varphi}
\renewcommand{\tilde}{\widetilde}
\renewcommand{\mapsto}{\longmapsto}

\renewcommand{\epsilon}{\varepsilon}

\newlist{steps}{enumerate}{1}
\setlist[steps, 1]{label = Step \arabic*:}

\newtheorem{theo}{Theorem}[section]
\newtheorem{prop}[theo]{Proposition}
\newtheorem{lem}[theo]{Lemma}
\newtheorem{rem}[theo]{Remark}

\newtheorem{definition}[theo]{Definition}

\numberwithin{equation}{section}

\title{Inevitable monokineticity of strongly singular alignment}
\author{Michał Fabisiak and Jan Peszek}
\date{\today}

\address{Institute of Applied Mathematics and Mechanics, University of Warsaw, ul. Banacha 2, 02-097 Warszawa, Poland}
\email{m.fabisiak@mimuw.edu.pl, j.peszek@mimuw.edu.pl}

\begin{document}

\thanks{\textbf{Acknowledgement:} The paper has been partly supported by the Polish National Science Centre’s Grant No. 2018/31/D/ST1/02313 (SONATA). JP's work has been additionally supported by the Polish National Science Centre's Grant No. 2018/30/M/ST1/00340 (HARMONIA)}

\subjclass{35Q70,
35D30, 
35L81, 
35Q83 }

\keywords{Cucker-Smale model, Euler-alignment system, weak solutions, monokinetic solutions, mean-field limit, singular interactions}

\maketitle

\begin{abstract}
    We prove that certain types of measure-valued mappings are monokinetic i.e. the distribution of velocity is concentrated in a Dirac mass. These include weak measure-valued solutions to the strongly singular Cucker-Smale model with singularity of order $\alpha$ greater or equal to the dimension of the ambient space. Consequently, we are able to answer a couple of open questions related to the singular Cucker-Smale model. First,  we prove that weak measure-valued solutions to the strongly singular Cucker-Smale kinetic equation are monokinetic, under very mild assumptions that they are uniformly compactly supported and weakly continuous in time. This can be interpreted as a rigorous derivation of the macroscopic fractional Euler-alignment system from kinetic Cucker-Smale equation without the need to perform any hydrodynamical limit.
    This suggests superior suitability of the macroscopic framework to describe large-crowd limits of strongly singular Cucker-Smale dynamics.
    Second, we perform a direct micro- to macroscopic mean-field limit from the Cucker-Smale particle system to the fractional Euler-alignment model. This leads to the final result -- existence of weak solutions to the fractional Euler-alignment system with almost arbitrary initial data in $\R^1$, including the possibility of vacuum. Existence can be extended to $\R^2$ under the a priori assumption that the density of the mean-field limit has no atoms.
\end{abstract}


\section{Introduction}
Cucker-Smale (CS) model \cite{CS-07} describes the motion of an ensemble of self-propelled particles, which cooperate in an attempt to align their velocities and flock. One of the defining factors of the CS model is the communication weight which controls the intensity of the interactions in terms of the distances between the particles. Classically, it is assumed to be bounded and Lipschitz continuous. In \cite{HL-09, P-14, P-15, CCMP-17}, a number of interesting phenomena exhibited by the CS model with singular communication was discovered, such as finite-time alignment and sticking of the particles or unconditional collision-avoidance, see also \cite{M-18, YYC-20}. Among them is the discovery of the Euler-type structure of the problem.  However the basic question of well-posedness, and even existence, for the singular CS model on the kinetic and hydrodynamical level remains open in many vital cases. It indicates that perhaps the very choice of a suitable framework to describe large-crowd singular CS dynamics  is the true question that has to be addressed.

Consider the microscopic CS model of $N$ particles following the ODE system
\begin{align}\label{micro}
    \begin{cases}
    \dot{x}_i = v_i, \\
    \dot{v}_i = \displaystyle\frac{1}{N}\displaystyle\sum\limits_{i\neq j=1}^N \psi(|x_i-x_j|) (v_j-v_i), \\
(x_i(0), v_i(0))= (x_{i0}, v_{i0}) \in \R^{2d},
    \end{cases}
\end{align}
where $x_i(t)$ and $v_i(t)$ are the position and velocity of the $i$th particle at the time $t\geq 0$, respectively. We fix our attention on the strongly singular communication weight
\begin{equation}\label{psi}
  \psi(s) = s^{-\alpha},\quad \alpha\geq d.
\end{equation}

As the number of particles increases to infinity, system \eqref{micro} is replaced by the following mesoscopic CS kinetic equation \cite{HL-09, CZ-21, MP-18, CC-21}
\begin{align}\label{meso}
    \begin{cases}
    \partial_t \mu + v \cdot \nabla_x \mu + \div_v \left( F(\mu) \mu \right) = 0, \quad x\in \R^d, v \in \R^d\\
    F(\mu)(t,x,v) := \displaystyle\int_{\R^{2d}} \psi(|x-y|)(w-v)\mu(t,y,w) \dd w \dd y,
     \end{cases}
\end{align}
where $\mu = \mu(t,x,v)$ is the distribution of the particles that at the time $t\geq 0$ are passing through position $x\in\R^d$ with velocity $v\in\R^d$.

Our goal is to prove that, due to the strong singularity $\alpha\geq d$, any sensible weak formulation of \eqref{meso} inevitably leads to the monokineticity of measure-valued solutions, meaning that for a.a. $t$ and $x$ the measure $\sigma_{t,x}(\cdot) = \mu(t,x,\cdot)$ is concentrated in a Dirac delta. In fact we prove a general theorem ensuring monokineticity of solutions to kinetic equations with dissipating kinetic energy, such as equations related to granular gases \cite{JR-16}. Such a proposition has three heavy implications, which can be viewed as our main results. First, it indicates that the kinetic setting \eqref{meso} is not necessarily suitable to describe strongly singular CS interactions, as it reduces to the monokinetic case -- leading automatically to the macroscopic scale and to the fractional Euler-aligmnent system. Such a result can be viewed as the derivation of the macroscopic Euler-alignment system from the kinetic Cucker-Smale equation, but -- unlike \cite{KMT-15, FK-19, PS-17} -- without the need to take the hydrodynamical limit. Second with $\alpha\in[d,2]$, we prove that attempting to pass to the mean-field limit from \eqref{micro} to \eqref{meso}, again, leads to monokineticity, which enables us to directly pass from microscopic to macroscopic description. This yields the third result -- we construct weak measure-valued solutions to the fractional Euler-alignment system initiated in any compactly supported initial datum in the whole Euclidean space of dimension 1. Previous results provided existence either with assumptions related to the critical threshold \cite{tan_euler-alignment_2020-1, HT-17, CCTT-16, TT-14}, with positive solutions in the whole space \cite{DMPW-19}
 or in the torus \cite{DKRT-18, ST-17, ST-18, STII-17, S-19}. It is noteworthy that our approach admits solutions with general initial data and vacuum.  The 1D existence can be conditionally extended to dimension 2, with $\alpha=2$, and with the a priori assumption that the mean-field limit is non-atomic.

\subsection{Monokineticity} In order to  rigorously introduce the notion of monokineticity that we use throughout the paper, we first need to invoke the disintegration theorem (cf. Theorem \ref{disintegration} in the appendix). Let $\mu = \mu(t,x,v)$ be a finite Radon measure on $[0,T]\times\R^{2d}$ and assume that the $t$-marginal of $\mu$ is the 1D Lebesgue measure on $[0,T]$. Then by Theorem \ref{disintegration} there exists a unique family $\{\mu_t\}_{t\in[0,T]}$ of probability measures on $\R^{2d}$ such that $\mu = \mu_t\otimes \lambda^1(t)$ (cf. the notation in  \eqref{disint_form}). Following this disintegration, throughout the paper we shall identify $\mu$ with the family $\{\mu_t\}_{t\in[0,T]}$ and write
\begin{equation}\label{mufamily}
     \mu = \{\mu_t\}_{t\in[0,T]}\quad \mbox{which we define as}\quad \mu(t,x,v) = \mu_t(x,v)\otimes\lambda^1(t).
\end{equation}

\noindent
For a.a. $t\in[0,T]$ each probability measure $\mu_t$ can be further disintegrated with respect to its $x$-marginal
$$ \rho_t(x):= \int_{\R^d}\dd \mu_t(x,v), $$
which we interpret as the local density. Thus, again in the sense of \eqref{disint_form}, we have
\begin{equation}\label{mudis}
    \mu(t,x,v) = \sigma_{t,x}(v)\otimes\rho_t(x)\otimes \lambda^1(t),
\end{equation}
and for a.a. $t\in[0,T]$ and $\rho_t$-a.a. $x$ the measure $\sigma_{t,x}$ is a probability measure on $\R^d$, which governs the local velocity distribution within $\mu$.

\bigskip
\begin{mdframed}
We say that $\mu$ is {\it monokinetic} if for a.a. $t\in[0,T]$ and $\rho_t$-a.a. $x$ the measure $\sigma_{t,x}$ is a Dirac delta concentrated in some $u(t,x)\in\R^d$ i.e.
$$ \mu(t,x,v) = \delta_{u(t,x)}(v)\otimes\rho_t(x)\otimes\lambda^1(t). $$
Then $(t,x)\mapsto u(t,x)$ is a function defined a.e. with respect to $\rho_t(x)\otimes\lambda^1(t)$, which plays the role of velocity.
\end{mdframed}

\bigskip
\noindent
In order to better understand our main theorem let us begin by recalling \cite{JR-16}, by {\sc Jabin and Rey}, wherein the authors prove that if $\mu$ is a weak solution to the equation

\begin{equation}\label{jabineq}
    \partial_t\mu + v\partial_x\mu = -\partial_{vv} m, 
\end{equation} 

\medskip
\noindent
then it is monokinetic. Here $m$ is any measure and it is assumed that $\mu$ is BV in time with values in some negative space (e.g. it is measure-valued). Such a framework is too restrictive for the purpose of application to the singular CS model \eqref{micro} and \eqref{meso}. We need to perform the following adjustments. First, we need our result to be applicable in a multi-D situation. Second, our prospective mean-field approximation is not uniformly continuous in time. In fact, as hinted in our very notion of monokineticity, our version of $\mu$ is defined a.e. and can have discontinuities. And lastly, the structure of our variant of the right-hand side of \eqref{jabineq} is more complex, namely we have

$$ \partial_t\mu + v\cdot\nabla_x\mu = Q(\mu), $$

\medskip
\noindent
where $Q(\mu)$ comes from the singular term $F(\mu)$ in \eqref{meso}, which, as we will show soon in Definition \ref{weakkinetic}, can only be tested with functions $\phi$ with uniformly Lipschitz continuous $(x,v)\mapsto \nabla_v\phi(t,x,v)$. In particular, the dependence on $x$ is more subtle compared to the original problem \eqref{jabineq}. These two issues: the lack of time-continuity and less regular right-hand side force us to operate with different assumptions, which are not easy to state in a straightforward way. Thus, we opt to introduce our assumptions in a rather abstract fashion. We shall assume that our family $\{\mu_t\}_{t\in[0,T]}$ in \eqref{mufamily} is uniformly compactly supported and
\begin{description}    \item[(MP) Locally mass preserving] For any set $C\subset \R^{d}$ and all $t_0\leq t\in[0,T]$ we have
$$ \rho_t(\overbar{C} + (t-t_0)\overbar{B(M)})\geq \rho_{t_0}(C), $$
where $M$ is the uniform bound of $v$-support of $\mu_t$, and $B(M)$ is a ball centered at $0$ with radius $M$, while $\overbar{C}$ denotes the closure of $C$.
    \item[(SF)  Steadily flowing]  There exists a full measure set $A\subset[0,T]$, such that for all $t_0\in A$ and any smooth compactly supported function $0\leq\phi\in C_c^\infty(\R^d)$ the family of finite Radon measures $\rho_{t_0,t}[\phi]$ defined for any Borel set $B\subset\R^d$ as
$$ \rho_{t_0,t}[\phi](B) := \int_{B\times\R^d}\phi(v)\dd T^{t_0,t}_\#\mu_t $$
is narrowly continuous at $t=t_0$, as a function of $t\in[0,T]$ restricted to $A$.
Here  $T^{t_0,t}_\#\mu_t$ is the pushforward measure of $\mu_t$ along $T^{t_0,t}(x,v) = (x-(t-t_0)v,v)$. Pushforward measure and the narrow topology are defined and discussed in the notation at the end of the introduction. We emphasise that the function $\rho_{t_0,t}[\phi]$ may be discontinuous everywhere; we require continuity after restriction to $A$, similar to how monotone functions are continuous after restriction to a full measure set.
\end{description}

We are now able to state our main theorem.

\begin{theo}\label{main1}
Suppose that the family of uniformly compactly supported probability measures $\{\mu_t\}_{t\in[0,T]}$ on $\R^{2d}$
is locally mass preserving and steadily flowing. Assume further that
\begin{align}\label{main1-as1}
    \int_0^T D^\alpha[\mu_t]\dd t:= \int_0^T \int_{\R^{4d}\setminus \{x=x'\}} \frac{|v-v'|^{\alpha+2}}{|x-x'|^\alpha}\dd [\mu_t\otimes\mu_t] \dd t <\infty, \qquad \alpha\geq d.
\end{align}
Then $\mu$ is monokinetic.
\end{theo}

\begin{rem}\rm
    The quantity $D^\alpha[\mu_t]$ in \eqref{main1-as1} is designed to be reminiscent of similar quantities in \cite{JR-16}. From our perspective, the key observation is that it is upper-bounded by the enstrophy appearing in the weak formulation of \eqref{meso}, cf. Definition \ref{weakkinetic}.    
\end{rem}

\begin{rem}\rm
    The idea behind assumptions (MP) and (SF) is that they replace the BV time-regularity required in \cite{JR-16}. Intuitively, we require the flow to conserve mass (or dissipate it in a controlled manner). However the only concept of the flow that we can reasonably use is in the sense of the pushforward measure $T^{t_0,t}_\#\mu_t$, which transforms only the position, leaving velocity intact. It is deliberately so, since we cannot control whole characteristics, due to the insufficient regularity of the velocity field. 
\end{rem}

\subsection{Monokineticity and the Euler-alignment system}
Following the disintegration introduced in the previous section we define the local quantities related to $\mu$. These are
\begin{equation}\label{localquant}
    \rho_t(x) := \int_{\R^d_v} \dd\mu_t(x,  v), \quad u(t,x) = \int_{\R^d_v}v\dd\sigma_{t,x}(v),
\end{equation}
which we refer to as the local density and the local velocity, respectively. Here the integration happens only with respect to $v\in\R^d$, which we emphasise by having $\R^d_v$ as the domain of integration. Moreover if $\rho_t$ and $u$ are regular then they solve the system
\begin{align}\label{macro_1}
    \begin{cases}
    \partial_t \rho + \div_x (\rho u) = 0 \\
    \partial_t (\rho u) + \div_x (\rho u \otimes u) = \displaystyle\int_{\R^d} \psi(|x-y|)(u(y) - u(x)) \rho(y)\rho(x) \dd y + P,
     \end{cases}
\end{align}
where the pressure term $P=(p_{ij})$ is given by
$$ p_{ij}=-\div_x\left(\int_{\R^d_v} (v_i - u_i(t,x))(v_j-u_j(t,x)) \dd \mu_t(x,v)\right)\qquad \mbox{see e.g. \cite{HT-08, C-19}}. $$
Note that in \eqref{macro_1} everything except $P$ is hydrodynamical or macroscopic i.e. depends only on $t$ and $x$. There are many ways to close system \eqref{macro_1} on the hydrodynamical level, two of which are related to the so-called Maxwelian  and monokinetic ansatz (see, for instance, \cite{S-2021-Surv}). However, in light of Theorem \ref{main1}, monokineticity of $\mu$ is ensured and does not need to be assumed as an ansatz. In such a case any solution to \eqref{meso} which is monokinetic reduces to a solution of the fractional Euler-alignment system
\begin{align}\label{macro}
    \begin{cases}
    \partial_t \rho + \div_x (\rho u) = 0 \\
    \partial_t (\rho u) + \div_x (\rho u \otimes u) = \displaystyle\int_{\R^d} \psi(|x-y|)(u(y) - u(x)) \rho(y)\rho(x) \dd y.
     \end{cases}
\end{align}
It remains true even for weak measure-valued solutions as long as they are narrowly continuous and uniformly compactly supported. This leads to the second main result of the paper.

\begin{theo}\label{main2}
Let $\alpha\geq d$ and $\mu$ be a weak measure-valued solution to \eqref{meso} in the sense of Definition \ref{weakkinetic} (see preliminaries below). Then $\mu$ is monokinetic and its local quantities \eqref{localquant} satisfy the Euler-alignment system \eqref{macro} in the weak sense of Definition \ref{weakeuler} (see preliminaries below).
\end{theo}

\begin{rem}\rm
The above theorem should be partially seen as a nonexistence result. For example, one could ask what happens if the initial datum $\mu_0$ is drastically non-monokinetic. For $$\mu_0(x,v) = \frac{1}{2}\delta_0(x)\otimes(\delta_{v_1}(v)+\delta_{v_2}(v))$$ we have two particles of mass $1/2$ situated at $0$ with two distinct initial velocities $v_1\neq v_2$. In such a case we claim that due to the weak continuity of the solution these two particles separate, leading to an immediate blowup of the enstrophy, see Proposition \ref{relaxdef}. In conclusion, the solution does not exist. In particular, we do not claim that a solution initiated in non-monokinetic data instantly becomes monokinetic.
\end{rem}

\begin{rem}\rm
We emphasise that the monokineticity result in Theorem \ref{main1} is necessary for Theorem \ref{main2} and cannot be replaced by previous results such as \cite{JR-16}, due to the more complex singular term $F(\mu)$ in \eqref{meso}.    
\end{rem}

In the last decade, Eulerian alignment dynamics has been derived rigorously from mesoscopic CS-type models via hydrodynamical limits. In \cite{KMT-15}, {\sc Karper, Mellet and Trivisa} derived the Euler-alignment system from the non-symmetric Motsch-Tadmor interactions \cite{MT-11}. Then, for singular CS model the limit was performed by {\sc Poyato and Soler} in \cite{PS-17}, with the latest contribution, for a wide class of non-singular CS-type models -- by {\sc Figalli and Kang} in \cite{FK-19} . The key difference between these results and ours is that we  exploit the singularity of interactions to derive the Euler-alignment model without a need to take the limit. In fact, any solution to the strongly singular CS kinetic equation -- if it is weakly continuous and uniformly compactly supported -- is a weak solution to the fractional Euler-alignment system. 



\subsection{Mean-field limit from micro- to macroscopic CS model}\label{intro:mf}
The relation between monokinetic solutions to the kinetic equation \eqref{meso} and solutions to the Euler-alignment system \eqref{macro} serve as a basis for the following approach. Solutions to the kinetic CS equation can often be derived as a mean-field limit of solutions to the particle system \eqref{meso}, see  \cite{HL-09, HT-08, CCR-2011}. In the case of singular communication weight it was done in the weakly singular case $\alpha\in(0,\frac{1}{2})$ in \cite{MP-18} and using the so-called first order reduction in \cite{CZ-21} (1D case) and \cite{PP-22-1-arxiv, PP-22-arxiv} (modified multi-D case). Here an opportunity arises to perform a mean-field limit from microscopic CS system to monokinetic solutions of the kinetic equation in the strongly singular case $\alpha\in[d,2]$. Thus, by Theorem \ref{main2}, it actually constitutes a direct micro- to macroscopic mean-field limit. However, execution of such a strategy is delicate since, due to high singularity, it is difficult (or impossible) to ensure that the mean-field approximation has a suitably convergent subsequence. Here lies the crux of the issue: we do not aim for the kinetic equation \eqref{meso} but for the macroscopic Euler-alignment system \eqref{macro} (since by Theorem \ref{main2} they are expected to coincide anyway). The macroscopic system caries signifficantly less information than the kinetic one. It turns out that the mean-field approximation has a convergent subsequence, whose limit $\mu$ satisfies the assumptions of Theorem \ref{main1}, and thus it is monokinetic. Moreover the local quantities \eqref{localquant} of $\mu$ satisfy the Euler-alignment system \eqref{macro}, even if $\mu$ itself does not solve the kinetic equation \eqref{meso}. Thus, we arrived at the third and last main result.

\begin{theo}\label{main3}
    Let $T>0$, $\rho_0\in{\mathcal P}(\R^{d})$ be compactly supported and $u_0:\R^d_x \rightarrow \R^d$ belong to $L^\infty(\rho_0)$. Moreover assume $\alpha \in [d,2]$, $d\in\{1,2\}$.
    Then there exists a mean-field approximative sequence $\mu^N$ of weak solutions to \eqref{meso}, converging narrowly to a monokinetic measure $\mu$, which satisfies the following assertions.
    
    \begin{enumerate}[label=(\Alph*)]
        \item If $d=1$ and $\alpha\in[1,2)$, then the local quantities $(\rho, u)$ associated with $\mu$ (see \eqref{localquant}) constitute a weak solution to the Euler-alignment system \eqref{macro} in the sense of Definition \ref{weakeuler}.
        \item If $d\in\{1,2\}$ and $\alpha=2$, then the local quantities $(\rho, u)$ associated with $\mu$ (see \eqref{localquant}) constitute a weak solution to the Euler-alignment system \eqref{macro} in the sense of Definition \ref{weakeuler}, provided that the local density $\rho$ is non-atomic for a.a. times $t\in[0,T]$.
    \end{enumerate}

\end{theo}

\begin{rem}[Necessity of assumptions and possible generalisations]\rm
    Theorem \ref{main3} introduces the assumption that $\alpha\leq 2$. It is related to the well-known quantity $$\int_0^T\int_{\R^{4d}}|x-x'|^{2-\alpha}\dd[\mu_t\otimes\mu_t]\dd t$$
    that emerges frequently in our calculations, which for $\alpha\leq 2$ is automatically manageable. Interestingly, its microscopic equivalent is known to be bounded, depending on initial data, even for $\alpha> 2$, see \cite{CCMP-17}. Thus it may be possible to generalise our results for $\alpha$ and $d$ greater than $2$.
\end{rem}



To the best of our knowledge, there are no prior results on existence of weak (or strong) solutions to the fractional Euler-alignment system with general initial data admitting vacuum in Euclidean space (either in 1D or multi-D case). The micro- to macroscopic mean-field limit and the monokineticity on the mesoscopic level was tackled recently by {\sc Carrillo and Choi} \cite{CC-21}, where the authors prove that if a strong solution to the Euler-alignment system with regular communication exists then it can be obtained by a direct micro- to macroscopic mean-field limit, which they support by providing also small data existence in such a scenario. The difference is that our mean-field limit operates with lesser regularity and does not require any prior knowledge on existence for the Euler-alignment system (in fact we use the mean-field limit to prove existence).

Let us fit our results into the larger landscape of the Euler-alignment system. In \cite{tan_euler-alignment_2020-1, HT-17, CCTT-16, TT-14} the issue of well-posedness for the Euler-alignment system with smooth communication was tackled using the critical threshold framework reminiscent of the works on the porous medium \cite{HZ-16} and quasi-geostrophic \cite{BD-03} equation. In the case of strongly singular fractional Euler-alignment system, a similar approach resulted in existence of strong solutions with initial data bounded away from vacuum in the 1D torus \cite{DKRT-18, ST-17, ST-18, STII-17}. Other directions include well-posedness for small initial data (or local-in-time) in multi-D torus \cite{S-19, C-19, CTT-2021} and the multi-D Euclidean space \cite{DMPW-19}. The issue of well-posendess of the fractional Euler-alignment system in multi-D with general initial data is mostly open. This and more can be found in the surveys  \cite{CHZ-17, MMPZ-19}. More recent directions of research include the study of singularity formation for the Euler-alignment system \cite{AC-21},  multi-D alignment dynamics with unidirectional velocities \cite{SL-2022, LLTS-2022} and models with various alternative types of communication such as topological \cite{ST-2020} and density-induced \cite{MMP-20} interactions, as well as interactions enhanced by anticipation \cite{ST-21}. Some of the recent developments can be found in the surveys \cite{ABFHKPPS-2019-Surv, S-2021-Surv}.






The remainder of the paper is organised as follows. In Section \ref{sec:prelim} we introduce the necessary preliminaries, including information on the particle system \eqref{micro} and weak formulations. Section \ref{sec:monokin} is dedicated to the proof of Theorem \ref{main1} and Theorem \ref{main2}. In Section \ref{sec:compact} we prove a handful of results on compactness related to the mean-field limit and in Section \ref{sec:mf} we prove that the mean-field limit is monokinetic and satisfies the Euler-alignment system, thus proving our final result, Theorem \ref{main3}. Some more tedious or simple lemmas and proofs are relegated to the appendix.

\subsection*{Notation}
Throughout the paper, by ${\mathcal P}(\R^d)$ we denote the space of probability measures  on $\R^d$ and by ${\mathcal M}(\R^d)$ and ${\mathcal M}_+(\R^d)$ -- Radon signed  and nonnegative measures on $\R^d$, respectively. We use two main topologies on measures. First, for a real-valued, possibly signed (or, with some adjustment, vector-valued) measure $\mu$, $||\mu||_{TV}$ denotes the total variation norm of $\mu$, inducing the strong topology. Second, we use the weak topology. Since all of the measures considered throughout the paper are (uniformly) compactly supported, many weak topologies on Radon measures coincide. We will mainly use the bounded-Lipschitz distance $d_{BL}$ (see Definition \ref{bddLip}). For compactly supported measures it induces the narrow topology (with convergence tested by bounded-continuous functions), as well as the weak-* topology in the sense of functional analysis. 
We will denote all of the above convergences as 
$$\mu_n\rightharpoonup\mu $$
and refer to it as narrow convergence.

For a Borel-measurable mapping $F:\R^d \rightarrow \R^n$ and a Radon measure $\mu$ we define the \textit{pushforward measure} $F_{\#}\mu$ by
    \begin{equation*}
    F_{\#}(\mu)(B):= \mu\left( F^{-1}(B) \right) \quad \text{for Borel sets } B.
    \end{equation*}
    The main property of the pushforward measure,  used throughout the paper, is the change of variables formula
    \begin{equation}\label{changeofvariables}
        \int_{\R^n} g \dd (F_{\#}\mu) = \int_{\R^d} g \circ F \dd \mu,
    \end{equation}
    whenever $g\circ F$ is $\mu$-integrable.

For any Radon measure $\rho$, by $L^p(\rho)$ we define the space of all $\rho$-measurable functions $f$, defined on the support of $\rho$, with $\rho$-integrable $|f|^p$ and the standard generalization as $p=\infty$. By $C(\R^n; \R^m)$ we denote the space of continuous functions on $\R^n$ with values in $\R^m$; we use a similar notation for other classical spaces -- $L^p$, $C_b$ (bounded continuous functions), $C_c^\infty$ (smooth compactly supported functions), $C^k$ (functions with up to $k$th continuous derivative); we abbreviate $C(\R^n)$ if $\R^m = \R$.

Finally we adopt a handful of notational conventions. We write  $\phi'$ if the  function $\phi$ depends on an alternative set of arguments $x'$ and/or $v'$, for example for a given $\phi=\phi(x,v)$ one would have $\phi'= \phi(x',v')$. In order to clarify the variable over which we integrate we will sometimes write  $\R^d_x$ and $\R^d_v$ with natural conclusion that $x\in \R^d_x$ and $v\in \R^d_v$. We do it for example in \eqref{localquant}.
   For $a,b \in \R$ we write $a \lesssim b$ if there exists a positive constant  $C\in\R$, independent of relevant parameters, such that $a\le C\ b$. Similarly, we write $a\approx b$ whenever two positive constants $C_1$, $C_2$ exist such that $C_1\ b \le a \le C_2\  b$. 
    
Whenever we use the notation ''a.e.'' without specifying the measure -- we always mean almost everywhere with respect to the 1D Lebesgue measure.

\section{Preliminaries}\label{sec:prelim}



In this section we introduce the main mathematical tools utilised in the paper, including the information on the CS particle system and the weak formulations for the meso- and macroscopic system. We start by discussing some properties of the microscopic, singular CS model.

\subsection{Microscopic Cucker-Smale particle system}\label{sec:micro}
The question of well-posedness for the strongly singular CS particle system \eqref{micro} has been studied and answered in \cite{CCMP-17}. This issue being crucial from the
point of view of our further consideration, we present an extract of some necessary information.
\begin{theo}[\cite{CCMP-17}]\label{exist_micro}
Let $d \ge 1$ and $\alpha \ge 1$. Suppose that initial data is non-collisional, i.e. it satisfies 
\begin{equation*}
    x_{i0}\neq x_{j0} \quad \text{for } 1\le i \neq j \le N.
\end{equation*}
Then system \eqref{micro} admits a unique smooth solution. Moreover, the trajectories of this solution remain non-collisional, i.e.
\begin{equation*}
    x_i(t) \neq x_j(t) \quad \text{for } 1\le i \neq j \le N, \quad t\ge 0.
\end{equation*}
\end{theo}

\noindent
We need additional two properties of the CS system, related to the propagation of the maximal velocity and position and the energy equality. Both of these results are classical (since by Theorem \ref{exist_micro} the trajectories are smooth) and are well-known in the literature cf. \cite{P-14} or \cite{CCH-14}.

\begin{prop}[Propagation of position and velocity]\label{flock}
Let $(x_i(t), v_i(t))_{i=1}^N$ be a solution to \eqref{micro} in $[0,T]$. Then we have 
\begin{equation*}
    \sup_{t\geq 0}\sup\limits_{1\le i \le N} |v_i(t)| \le \sup\limits_{1\le i \le N} |v_{i0}|, \quad 
    \sup_{t\geq 0}\sup\limits_{1\le i \le N} |x_i(t)| \le \sup\limits_{1\le i \le N} |x_{i_0}| + t\sup\limits_{1\le i \le N} |v_{i0}|.
\end{equation*}
\end{prop}
\noindent The above implies the existence of a positive constant $M$ such that

\begin{equation}\label{boundM}
    \sup_{t\geq 0}\sup\limits_{1\le i \le N} |v_i(t)| < M,\quad \sup_{t\geq 0}\sup\limits_{1\le i \le N} |x_i(t)|< (T+1)M.
\end{equation}

\begin{prop}[Energy equality]\label{eneq}
 Let $(x_i(t), v_i(t))_{i=1}^N$ be a solution to \eqref{micro} in $[0,T]$. Then for all $t\in[0,T]$, we have
 \begin{equation*}
     \frac{1}{N^2}\int_0^t\sum_{i\neq j = 1}^N|v_i(s)-v_j(s)|^2|x_i(s)-x_j(s)|^{-\alpha} \dd s = \frac{1}{N}\sum_{i=1}^N|v_i(0)|^2 - \frac{1}{N}\sum_{i=1}^N|v_i(t)|^2 .
 \end{equation*}
\end{prop}



\subsection{Mesoscopic kinetic CS model}
We begin this section by introducing the bounded-Lipschitz distance. As explained in the notation, since we consider only compactly supported measures, the precise choice of the weak topology becomes mostly a matter of taste.  We use the narrow topology characterised by convergence tested with bounded-continuous functions. For compactly supported measures it is metrizable by the bounded-Lipschitz distance defined as follows.

\begin{definition}[\textbf{Bounded-Lipschitz distance}]\label{bddLip}
Let $\mu, \nu \in \mathcal{M}_+(\R^d)$ be two finite Radon measures. The bounded Lipschitz distance $d_{BL}(\mu, \nu)$ between them is given by
$$
d_{BL}(\mu,\nu) := \sup\limits_{\phi \in \Gamma_{BL}} \left| \int_{\R^d} \phi \dd \mu - \int_{\R^d} \phi \dd \nu \right|,
$$
where $\Gamma_{BL}$ is the set of admissible test functions
$$
\Gamma_{BL} := \left\{ \phi: \R^d \longmapsto \R : ||\phi||_{\infty} \le 1,\ [\phi]_{Lip}:= \sup\limits_{x\neq y}\frac{|\phi(x)-\phi(y)|}{|x-y|}\le 1 \right\}.
$$
\end{definition}

\begin{rem}\rm\label{rem:compact}
    We note that the space $({\mathcal M}(\Omega), d_{BL})$ of uniformly compactly supported signed Radon measures, with supports contained in a compact $\Omega\subset \R^d$ is isomorphic to $C(\Omega)^*$, the dual of continuous functions on $\Omega$. Thus, by Banach-Alaoglu theorem, any subset of ${\mathcal M}(\Omega)$ bounded in the $TV$ topology is relatively compact in  ${\mathcal M}(\Omega)$ with narrow topology. It is noteworthy that the subspaces $({\mathcal M}_+(\Omega), d_{BL})$ of nonnegative Radon measures and $({\mathcal P}(\R^d),d_{BL})$ are complete, which means that strongly bounded sequences in these spaces are relatively compact with limits belonging to the respective space. We use this fact frequently in Section \ref{sec:compact}. Similar results can be obtained via Prokhorov's theorem by exploiting tightness of uniformly compactly supported measures.
\end{rem}

We proceed with the weak formulation for the kinetic CS equation \eqref{meso}. First let us denote the
\begin{itemize}
    \item \textbf{Diagonal set} $\Delta$ defined as
    \begin{equation}\label{diagonal}
    \Delta:= \{(x,v,x',v')\in\R^{4d}:\quad x=x'\}\quad\mbox{or}\quad \Delta:= \{(x,x')\in\R^{2d}:\quad x=x'\},
\end{equation}
depending on the context;
    \item \textbf{Kinetic energy} $E[\mu_t]$ defined as
    $$
    E[\mu_t] := \int_{\R^{2d}} |v|^2 \dd \mu_t(x,v);
    $$
    \item \textbf{Energy dissipation rate (or enstrophy)} $D[\mu_t]$ defined as 
    $$
    D[\mu_t] := \int_{\R^{4d}\setminus\Delta} \frac{|v-v'|^2}{|x-x'|^{\alpha}} \dd\left[\mu_t(x,v) \otimes \mu_t(x',v')\right].
    $$
\end{itemize}
We stress that the diagonal set $\Delta$ may be of positive product measure $\mu_t\otimes\mu_t$ and therefore needs to be removed from the domain of integration, due to singularity of the integrand. Next, we introduce the weak formulation.
\begin{definition}[\textbf{Weak solution of the kinetic equation}]\label{weakkinetic}
For a fixed $0<T<\infty$ we say that $\mu = \{\mu_t\}_{t\in[0,T]} \in C\left( [0,T]; (\mathcal{P}(\R^{2d}), d_{BL}) \right)$ is a weak measure-valued solution of \eqref{meso} with compactly supported initial datum $\mu_0 \in \mathcal{P}\left(\R^{2d}\right)$ if the following assertions are satisfied.
\begin{enumerate}[label=(\roman*)]
    \item There exists a constant $M>0$ such that for a.a. $t\in[0,T]$ we have
    $$\text{spt}(\mu_t)\subset\subset (T+1)B(M)\times B(M)\subset \R^d_x\times\R^d_v.$$
  In other words, the $v$-support of $\mu_t$ is compactly contained in the ball $B(M)$, while the $x$-support of $\mu_t$ is compactly contained in $(T+1)B(M)$, cf. \eqref{boundM}.
    \item For each $\phi \in C^1([0,T]\times\R^{2d})$, compactly supported in $[0,T)$ with Lipschitz continuous $\nabla_v\phi$ (in $B(M)$), the following identity holds
            \begin{equation}\label{weakeqkinetic}
            \begin{split}
            - \int_{\R^{2d}}\phi(0,x,v) \dd \mu_0(x,v) &= \int_0^T \int_{\R^{2d}} \left(\partial_t \phi(t,x,v) + v\cdot \nabla_x \phi(t,x,v) \right) \dd \mu_t(x,v) \dd t\\
            &+ \frac{1}{2} \int_0^T \int_{\R^{4d}\setminus\Delta} \frac{\left(\nabla_v \phi(t,x,v) - \nabla_v \phi(t,x',v')\right)\cdot (v - v')}{|x-x'|^{\alpha}} \dd [\mu_t(x,v) \otimes \mu_t(x',v')] \dd t.
            \end{split}
            \end{equation}
            In particular, all integrals in the above equation are well defined.
\end{enumerate}
\end{definition}

\begin{rem}[On the singular term]\rm
    The singular term in \eqref{weakeqkinetic} follows the same principle as the bilinear form for the fractional Laplacian. Indeed, using \eqref{meso} and  Fubini's theorem for smooth enough functions $\mu$ we have
    \begin{align*}
        \frac{1}{2} \int_0^T \int_{\R^{4d}\setminus\Delta} \frac{\left(\nabla_v \phi(t,x,v) - \nabla_v \phi(t,x',v')\right)\cdot (v - v')}{|x-x'|^{\alpha}} \dd [\mu_t(x,v) \otimes \mu_t(x',v')] \dd t\\ = \int_0^T \int_{\R^{2d}}\mu(t,x,v)\nabla_v\phi(t,x,v)\cdot \int_{\R^{2d}\setminus\Delta} \frac{(v - v')}{|x-x'|^{\alpha}}\mu(t,x',v') \dd x' \dd v' \dd x\dd v \dd t\\
        = -\int_0^T \int_{\R^{2d}}\phi\ {\rm div}_v[ F(\mu)\mu] \dd x\dd v\dd t,
    \end{align*}
    and thus we recover the last term of \eqref{meso} tested by $\phi$.
\end{rem}

The crucial property of the weak formulation in Definition \ref{weakkinetic} is that it ensures that solutions dissipate kinetic energy, as stated in the following proposition.

\begin{prop}[Compactly supported solutions dissipate kinetic energy]\label{relaxdef}
For a fixed $0<T<\infty$ and $\mu_0 \in \mathcal{P}\left(\R^{2d}\right)$, let $\mu_t \in C\left( [0,T]; (\mathcal{P}(\R^{2d}), d_{BL}) \right)$ satisfy (i) and (ii) from Definition \eqref{weakkinetic}. Then $\mu$ satisfies the energy equality
    \begin{equation}\label{endisineq}
        \int_0^tD[\mu_s]\dd s = E[\mu_0] - E[\mu_t]
    \end{equation}
    for a.a. $t\in[0,T]$.
Conversely if \eqref{endisineq} holds then all terms in the weak formulation \eqref{weakeqkinetic} are well defined.
\end{prop}

As our final effort in this section we introduce the main tool linking the particle and the kinetic level.
\begin{definition}[\textbf{Atomic solution}]\label{empirical}
For a fixed $N\in\mathbb{N}$, let $(x_i^N(t), v_i^N(t))_{i=1}^N$ be a solution to the particle system \eqref{micro}. We define the atomic solution $\mu^N\in{\mathcal M}([0,T]\times\R^{2d})$ associated with $(x_i^N(t), v_i^N(t))_{i=1}^{N}$ as the family of empirical measures $\{\mu_t^N\}_{t\in[0,T]}$ (cf. \eqref{mufamily}) of the form 
\begin{equation*}
    \mu_t^N(x,v) := \frac{1}{N} \sum\limits_{i=1}^N \delta_{x_i^N(t)}(x) \otimes \delta_{v_i^N(t)}(v),\qquad \mbox{for all }t\in[0,T],
\end{equation*}
which is a solution to the kinetic equation \eqref{meso} in the sense of Definition \ref{weakkinetic}, see for instance \cite[Remark 22]{MP-18}.
\end{definition}

\medskip

\subsection{Macroscopic CS model and the hydrodynamical Euler-alignment system}
We turn our attention to the weak formulation for the Euler-alignment system \eqref{macro}, which from our point of view is a monokinetic reduction of the kinetic formulation in Definition \ref{weakkinetic}. Let $\mu$ be a weak measure-valued solution to \eqref{meso} in the sense of Definition \ref{weakkinetic}. Testing \eqref{weakeqkinetic} with $\phi(t,x,v) = \tilde{\phi}(t,x)$ we obtain the continuity equation
\begin{align*}
    0 &= \int_{\R^{2d}} \tilde{\phi}(t,x) \dd \mu_0(x,v) + \int_0^T \int_{\R^{2d}} \left( \partial_t \tilde{\phi}(t,x) + v \cdot \nabla_x \tilde{\phi}(t,x) \right) \dd \mu_t(x,v) \dd t\\
    &= \int_{\R^{d}} \tilde{\phi}(t,x) \dd \rho_0(x) + \int_0^T \int_{\R^{d}}  \partial_t \tilde{\phi}(t,x)\dd \rho_t(x) \dd t + \int_0^T \int_{\R^{d}} \nabla_x \tilde{\phi}(t,x) \cdot u(t,x) \dd   \rho_t(x) \dd t.
\end{align*}
Testing with $\phi(t,x,v) = v_i\tilde{\phi}(t,x) $, where $v=(v_1,...,v_d)$, we obtain the momentum equation
\begin{equation}\label{wypro}
\begin{split}
        -\int_{\R^{2d}} v_i \tilde{\phi}(t,x) \dd \mu_0 (x,v) &= \int_0^T\int_{\R^{2d}} \left( v_i \partial_t \tilde{\phi}(t,x) + v_i (v\cdot \nabla_x) \tilde{\phi}(t,x) \right) \dd \mu_t \dd t \\
    &+ \frac{1}{2} \int_0^T \int_{\R^{4d}\setminus\Delta} \frac{\left( \tilde{\phi}(t,x) - \tilde{\phi}(t,x')\right) (v_i - v'_i)}{|x-x'|^{\alpha}} \dd [\mu_t \otimes \mu_t] \dd t
\end{split}
\end{equation}
In the above formula we use the monokineticity of $\mu(t,x,v) = \delta_{u(t,x)}(v)\otimes \rho_t(x)\otimes \lambda^1(t)$ in order to deal with the convection
$$ \int_0^T\int_{\R^{2d}} \partial_{x_j}\tilde{\phi}(t,x)v_jv_i \dd \mu_t(x,v)\dd t =  \int_0^T\int_{\R^{d}} \partial_{x_j}\tilde{\phi}(t,x) u_j(t,x)u_i(t,x)\dd\rho_t(x) \dd t$$
and translate \eqref{wypro} into
\begin{align*}
    -\int_{\R^{2d}} u_i(0,x) \tilde{\phi}(t,x) \dd \rho_0 (x) &= \int_0^T \int_{\R^{d}}\left(  \partial_t u_i(t,x) \tilde{\phi}(t,x) + u_i(t,x)( u(t,x)\cdot \nabla_x) \tilde{\phi}(t,x) \right) \dd \rho_t(x) \dd t \\
    &+ \frac{1}{2} \int_0^T \int_{\R^{2d}\setminus\Delta} \frac{\left( \tilde{\phi}(t,x) - \tilde{\phi}(t,x')\right)(u_i(t,x) - u_i(t,x'))}{|x-x'|^{\alpha}} \dd [\rho_t(x) \otimes \rho_t(x')] \dd t.
\end{align*}

\noindent
Therefore the kinetic formulation according to Definition \ref{weakkinetic} for monokinetic $\mu$ formally reduces to the following weak formulation for the Euler-alignment system.
\begin{definition}[\textbf{Weak solution to the Euler-alignment system}]\label{weakeuler} 
For a fixed $0<T<\infty$ we say that the pair $(\rho, u)$ with $\rho_t\in C([0,T]; (\mathcal{P}(\R^d), d_{BL}))$ and $u\in L^1([0,T];L^1(\rho_t))\cap L^\infty(\rho)$ is a weak solution of \eqref{macro} with compactly supported initial data $\rho_0\in{\mathcal P}(\R^d)$ and $u_0\in L^\infty(\rho_0)$ if the following assertions are satisfied.
\begin{enumerate}[label=(\roman*)]
    \item There exists a constant $M>0$ such that for a.a. $t\in[0,T]$ we have
    \begin{equation*}
    {\rm spt}(\rho_t) \subset\subset (T+1) B(M) \subset \R^d_x\quad \mbox{and}\quad \|u(t,\cdot)\|_{L^\infty(\rho_t)} < M;
    \end{equation*}
    \item For a.a. $t \in[0,T]$
    $$
    \int_0^t \int_{\R^{2d}\setminus\Delta} \frac{|u-u'|^2}{|x-x'|^\alpha} \dd[\rho_s(x) \otimes \rho_s(x')] \dd s \le \int_{\R^{d}} |u(0,x)|^2 \dd \rho_0(x) - \int_{\R^{d}} |u(t,x)|^2 \dd \rho_t(x)
    $$
    \item For each $\phi\in C^1([0,T]\times\R^d)$ and each $\phi_d\in C^1([0,T]\times\R^d; \R^d)$, compactly supported in $[0,T)$,  the following  equations are satisfied
    \begin{equation}\label{weakmacro}
    \begin{split}
        \displaystyle\int_0^T \int_{\R^d} (\partial_t \phi + \nabla_x \phi \cdot u ) \dd \rho_t(x) \dd t &= - \displaystyle\int_{\R^d} \phi(0,x) \dd \rho_0(x) \\ 
        \displaystyle\int_{\R^{d}} u(0,x)\cdot \phi_d(0,x) \dd \rho_0(x) &+\displaystyle\int_0^T \int_{\R^d}\left( \partial_t  \phi_d\cdot u + u(u\cdot \nabla_x) \phi_d\right) \dd \rho_t \dd t \\
        =- \displaystyle\frac{1}{2} &\displaystyle\int_0^T \int_{\R^{2d}\setminus\Delta} \frac{(\phi_d(t,x) - \phi_d(t,x'))\cdot(u(t,x)-u(t,x'))}{|x-x'|^{\alpha}} \dd [\rho_t(x)\otimes \rho_t(x')] \dd t.
        \end{split}
    \end{equation}  
\end{enumerate}
\end{definition}

In order to validate the weak formulation in Definition \ref{weakeuler} one could compare it to the notion of strong solutions to the Euler-alignment system, which appear frequently in the literature. The weak-strong or weak-atomic (see \cite{MP-18}) uniqueness are beyond our current work, since it would require a lot of non-trivial computation related to the commutator estimates, as in the classical works on Euler equation. Similar problems arise, when one asks if Theorem \ref{main2} can be reversed, i.e. do weak solutions to the Euler-alignment system automatically produce a monokinetic solution to the CS equation. Such questions seem very interesting and we plan to tackle them in the future. Nevertheless, let us make o couple of readily available comments. 
First, if our solution originates from a kinetic solution in the sense of Definition \ref{weakkinetic} then it is clear that the momentum $u(t,x)\rho_t(x)$, treated as a function of $t\in[0,T]$ with values in $\R^d$-valued measures is narrowly continuous. Second, perhaps more interestingly, whenever $\rho$ is (at least locally) separated from $0$, the singular term in \eqref{weakmacro} is comparable to the bilinear form of the fractional laplacian of order $\gamma = \frac{\alpha - d}{2}$. Therefore one could expect $u$ to belong to the fractional Sobolev space $W^{\gamma,2}$, which further implies  H\" older continuity, provided that $\gamma> \frac{d}{2}$. Finally, it is easy to show that if $(\rho, u)$ constitute a regular classical solution to the Euler-alignment system then the family of measures $\mu_t(x,v) = \delta_{u(t,x)}(v)\otimes\rho_t(x)$ is a solution to the CS equation in the sense of Definition \ref{weakkinetic}, and thus Theorem \ref{main2} can be reversed at least in such a scenario.

\subsection{Measurability and reversing disintegration}\label{sec:redisint}
    In Definition \ref{weakeuler} and in the preceding formal derivation, we silently encounter an issue of measurability of various measures and functions related to the disintegration Theorem \ref{disintegration}. Namely, it is not necessarily clear at the first glance that the function $u$ defined as the velocity of the monokinetic measure $\mu$ belongs to the space $L^1([0,T];L^1(\rho_t))$ as required in Definition \ref{weakmacro}. It is bounded due to the uniformly compact support of $\mu$; the problem is with measurability. We argue as in \cite[Proposition 2.14]{PP-22-1-arxiv}, wherein condition (i) from Theorem \ref{disintegration} was explained to be equivalent to the Borel-measurability of the function $x_2\mapsto \mu_{x_2}$ as a function with values in ${\mathcal P}(\R^d)$ with the (metrizable) narrow topology. Then the key observation is that, recalling the disintegration \eqref{mudis}, we have
    $$ u(t,x) = \int_{\R^d_v} v\dd \sigma_{t,x}(v)\quad\mbox{for}\quad \sigma_{t,x}(v)=\delta_{u(t,x)}(v). $$
    Therefore the function $(t,x)\mapsto u(t,x)$ is a composition of the Borel-measurable function $(t,x)\mapsto \sigma_{t,x}$ and a continuous function $$({\mathcal M}(\R^d) \ \mbox{with narrow topology})\ni \sigma\mapsto \int_{\R^d_v}v\dd\sigma(v)\in\R$$
    which makes it Borel-measurable.
    
    A similar argument can be used to show that other measures are well-defined  as ''reversals'' of the disintegration in Theorem \ref{disintegration}. This includes, using the notation from \eqref{mudis}, the measure
    $$ \rho(t,x) = \rho_t(x)\otimes\lambda^1(t)\in {\mathcal M}_+([0,T]\times\R^{d}),$$
    which is defined by stitching together measures  $\rho_t(x) = \int_{\R^d_v}\dd\mu_t(x,v)$, for a.a. $t\in[0,T]$. By the narrow continuity of the product of measures, the quantities
    $$ [\mu_t(x,v)\otimes\mu_t(x',v')]\otimes\lambda^1(t),\qquad [\rho_t(x)\otimes\rho_t(x')]\otimes\lambda^1(t)$$
    are well defined as measures on $[0,T]\times \R^{4d}$ and on $[0,T]\times\R^{2d}$, respectively.
    Note that in both measures above, the first $\otimes$ stands for the classical product of measures in the sense of Fubini's theorem, while the second is defined through the disintegration Theorem \ref{disintegration}.


\section{Proof of the monokineticity}\label{sec:monokin}
In this section we prove Theorem \ref{main1} and apply it to weak solutions of \eqref{meso} to obtain Theorem \ref{main2}. 

\begin{proof}[Proof of Theorem \ref{main1}]
   Using representation \eqref{mudis}
our goal is to prove that for a.a. $t\in[0,T]$ and for $\rho_t$-a.a. $x\in\R^d$, the measure $\sigma_{t,x}$ is a Dirac mass, i.e.
\begin{equation}\label{main1-goal}
    \sigma_{t,x}(v) = \delta_{u(t,x)}(v)\quad \mbox{for}\quad u(t,x) = \int_{\R^d_v}v\dd \sigma_{t,x}(v)  .  
\end{equation}

\noindent
The proof follows in three steps concerning the non-atomic part of $\rho_t$ with $\alpha>d$ and separately with $\alpha=d$, and the atomic part of $\rho_t$. Therefore it is useful to define for a.a. $t\in[0,T]$,

\begin{equation}\label{atomandnot}
    \rho_t(x) = \widetilde{\rho}_t(x) + \sum_{n=1}^\infty \rho_n(t)\delta_{x_n(t)}(x),\quad
\widetilde{\mu}_t(x,v) = \sigma_{t,x}(v)\otimes \widetilde{\rho}_t(x),
\end{equation}
 where $\widetilde{\rho}_t$ is the non-atomic part of the measure $\rho_t$, while $\rho_n(t)$ and $x_n(t)$ denote the mass and position of $n$th atom at the time $t$. Measure $\widetilde{\mu}$ is defined through a ''reversal'' of the disintegration theorem, which is rigorous since $\widetilde{\rho_t}$ is a Borel measure and $(t,x)\mapsto \sigma_{t,x}$ is Borel-measurable as explained in Section \ref{sec:redisint}.

Recall that $\{\mu_t\}_{t\in[0,T]}$ is assumed to be uniformly compactly supported. For the sake of consistence, throughout the proof we shall use the notation from item (i) in Definition \ref{weakkinetic} and assume that the support of $\mu_t$ lies in $(T+1)B(M)\times B(M)$, even though $\{\mu_t\}_{t\in[0,T]}$ does not necessarily satisfy Definition \ref{weakkinetic}.

\medskip
\noindent
$\diamond$ {\sc Step 1.} {\it The non-atomic case with $\alpha> d$.}

\noindent
Interestingly enough, the proof in the non-atomic case follows only from assumption \eqref{main1-as1}. First, let us make the following crucial observation: since $\widetilde{\rho}_t$ is non-atomic, by Fubini's theorem, we have
$\widetilde{\rho}_t\otimes\widetilde{\rho}_t(\Delta) = 0 $, which together with formula \eqref{atomandnot} implies that
\begin{align*}
    D^\alpha[\widetilde{\mu}_t] &= \int_{\R^{2d}\setminus\Delta}\left(\int_{\R^{2d}}|v-v'|^{\alpha+2}\dd\left[\sigma_{t,x}(v)\otimes\sigma_{t,x'}(v')\right]\right) |x-x'|^{-\alpha} \dd\left[\widetilde{\rho}_t(x)\otimes\widetilde{\rho}_t(x')\right]\\
    &= \int_{\R^{2d}}\left(\int_{\R^{2d}}|v-v'|^{\alpha+2}\dd\left[\sigma_{t,x}(v)\otimes\sigma_{t,x'}(v')\right]\right) |x-x'|^{-\alpha} \dd\left[\widetilde{\rho}_t(x)\otimes\widetilde{\rho}_t(x')\right]\\
 & = \int_{\R^{4d}} \frac{|v-v'|^{\alpha+2}}{|x-x'|^\alpha} \dd \left[\widetilde{\mu}_t(x,v)\otimes\widetilde{\mu}_t(x',v')\right].
\end{align*}
Here $\Delta$ is the diagonal set defined in \eqref{diagonal}.
Thus, by assumption \eqref{main1-as1}, we have
\begin{align}\label{main1-crux}
  K_0:= \int_0^T\int_{\R^{4d}} \frac{|v-v'|^{\alpha+2}}{|x-x'|^\alpha} \dd \left[\widetilde{\mu}_t(x,v)\otimes\widetilde{\mu}_t(x',v')\right]\dd t <\infty.
\end{align}
Next, let us fix $\eta>0$ and introduce the following quantity
\begin{align*}
    D_\eta^\alpha[\widetilde{\mu}_t]:= \int_{\R^{4d}}\frac{|v-v'|^{\alpha+2}}{(|x-x'|+\eta)^\alpha}\dd\left[\widetilde{\mu}_t(x,v)\otimes\widetilde{\mu}_t(x',v')\right].
\end{align*}
Recalling the notation \eqref{main1-goal}, by Jensen's inequality we have, at a.a. $t\in[0,T]$
\begin{align*}
    \int_{\R^{2d}}\frac{|u(t,x) - u(t,x')|^{\alpha+2}}{ (|x-x'|+\eta)^{\alpha}} \dd \left[\widetilde{\rho}_t(x)\otimes\widetilde{\rho}_t(x')\right] &\leq D_\eta^\alpha[\widetilde{\mu}_t],\\
    \int_{\R^{3d}}\frac{|v - u(t,x')|^{\alpha+2} }{ (|x-x'|+\eta)^{\alpha}} \dd \left[\widetilde{\mu}_t(x,v)\otimes\widetilde{\rho}_t(x')\right] &\leq D_\eta^\alpha[\widetilde{\mu}_t].
\end{align*}
The above inequalities, together with the elementary inequality $(a+b)^\alpha\leq 2^\alpha(a^\alpha+b^\alpha)$, imply that
\begin{equation}\label{koniec1}
\begin{split}
    {\mathcal E}_\eta[\widetilde{\mu}_t] &:=\int_{\R^{3d}}\frac{|v - u(t,x)|^{\alpha+2}}{ (|x-x'|+\eta)^{\alpha}} \dd \left[\widetilde{\mu}_t(x,v)\otimes\widetilde{\rho}_t(x')\right]\\
    &= \int_{\R^{3d}}\frac{|v - u(t,x') + u(t,x')- u(t,x)|^{\alpha+2}}{ (|x-x'|+\eta)^{\alpha}} \dd \left[\widetilde{\mu}_t(x,v)\otimes\widetilde{\rho}_t(x')\right]\\
    &\leq 2^\alpha \int_{\R^{3d}}\frac{|v - u(t,x')|^{\alpha+2} }{ (|x-x'|+\eta)^{\alpha}} \dd \left[\widetilde{\mu}_t(x,v)\otimes\widetilde{\rho}_t(x')\right]\\
    &\quad + 2^\alpha\int_{\R^{2d}}\frac{|u(t,x) - u(t,x')|^{\alpha+2}}{ (|x-x'|+\eta)^{\alpha}} \dd \left[\widetilde{\rho}_t(x)\otimes\widetilde{\rho}_t(x')\right]
    \leq 2^{\alpha+1} D_\eta^\alpha[\widetilde{\mu}_t].
    \end{split}
\end{equation}

\noindent
Meanwhile we also have
\begin{align*}
    {\mathcal E}_\eta[\widetilde{\mu}_t] = \beta_\eta\int_{\R^{2d}} |v-u(t,x)|^{\alpha+2} K_\eta * \widetilde{\rho}_t(x)\dd \widetilde{\mu}_t(x,v),
\end{align*}
where
$$ K_\eta(x) := \beta^{-1}_\eta\frac{1}{(|x|+\eta)^\alpha},\quad \beta_\eta = \int_{\R^d} \frac{1}{(|x|+\eta)^\alpha}\dd x $$
are well defined, since $\alpha> d$. Combining \eqref{koniec1} with the disintegration \eqref{atomandnot}, and integrating in time, we obtain
\begin{align}\label{main1-p2}
    \int_0^T \int_{\R^{2d}} |v-u(t,x)|^{\alpha+2} K_\eta * \widetilde{\rho}_t(x)\dd \left[\sigma_{t,x}(v)\otimes \widetilde{\rho}_t(x)\right] \dd t \leq \beta_\eta^{-1} 2^{\alpha+1}\int_0^T D_\eta^\alpha[\widetilde{\mu}_t]\dd t\leq \beta_\eta^{-1}2^{\alpha+1}K_0.
\end{align}
Since the function $x\mapsto |x|^{-\alpha}$ is not integrable in any neighborhood of $0\in\R^d$, we have $\beta_\eta\to \infty$ as $\eta\to 0$, and thus, using \eqref{main1-crux}, both sides of inequality \eqref{main1-p2} converge to $0$ as $\eta\to 0$. By Fatou's lemma
\begin{align*}
    \int_0^T \int_{\R^d} \left(\int_{\R^{d}} |v-u(t,x)|^{\alpha+2}\dd\sigma_{t,x}(v)\right) \liminf_{\eta\to 0}(K_\eta * \widetilde{\rho}_t(x))\dd \widetilde{\rho}_t(x)\dd t = 0.
\end{align*}
From the above identity we conclude that for $\widetilde{\rho}_t\otimes \lambda^1(t)$-a.a. $(t,x)\in[0,T]\times\R^{d}$, such that 
\begin{equation}\label{main_eq1}
    \liminf_{\eta\to 0}(K_\eta * \widetilde{\rho}_t(x))>0
\end{equation}
 the support of $\sigma_{t,x}$ is concentrated in the set $\{v = u(t,x)\}$ (which means precisely that $\sigma_{t,x}(v)=\delta_{u(t,x)}(v)$). Thus the main remaining problem is the analysis of condition \eqref{main_eq1}. For a.a. $t$, we have
 \begin{equation}\label{main1-p3}
     K_\eta * \widetilde{\rho}_t(x) = \beta^{-1}_\eta\int_{\R^d}\frac{\dd \widetilde{\rho}_t(x')}{(|x-x'|+\eta)^\alpha}\gtrsim
     \eta^{-\alpha}\beta_\eta^{-1}\int_{B(x,\eta)}\dd \widetilde{\rho}_t(x')
 \end{equation}

 \noindent
 and our next goal is to lower-bound $\eta^{-\alpha}\beta_\eta^{-1}$. To this end we use the spherical coordinates to integrate
 \begin{align*}
     \eta^\alpha \int_{\R^d}\frac{\dd x}{(|x|+\eta)^\alpha} \approx \eta^\alpha \int_0^\infty(r+\eta)^{-\alpha}r^{d-1}\dd r = \eta^\alpha \int_\eta^\infty r^{-\alpha}(r-\eta)^{d-1}\dd r
 \end{align*}
 and  the right-hand side above can be expressed as a sum of terms of the form
 \begin{align*}
     \eta^\alpha\int_\eta^\infty r^{\theta-\alpha}\eta^{\theta'}\dd r,\qquad \theta + \theta' = d-1,
 \end{align*}
which integrated are equal (up to constant factors) to
  \begin{align*}
    -\eta^{\alpha+\theta'} \left(r^{\theta+1-\alpha}\Bigg|^{r=\infty}_{r=\eta}\right) = \eta^{\alpha+\theta'+\theta +1 -\alpha} = \eta^d.
\end{align*}
Therefore $\eta^\alpha\beta_\eta\approx \eta^d$, which, together with \eqref{main1-p3}, leads to
 \begin{align*}
     K_\eta * \widetilde{\rho}_t(x)\gtrsim \eta^{-d}\int_{B(x,\eta)}\dd\widetilde{\rho}_t(x').
 \end{align*}
Now, the support of $\widetilde{\rho}_t$ can be decomposed into two sets -- one, where $\widetilde{\rho}_t$ is absolutely continuous (with respect to the Lebesgue measure), and the second, where $\widetilde{\rho}_t$ is singular. By Lebesgue-Besicovitch differentiation theorem, on the first set we have
$$ \lim_{\eta\to 0}\eta^{-d}\int_{B(x,\eta)}\dd\widetilde{\rho}_t(x') = \lim_{\eta\to 0}\eta^{-d}\int_{B(x,\eta)} \widetilde{\rho}_t(x') \dd x' \approx\widetilde{\rho}_t(x)>0 $$
for $\widetilde{\rho}_t$-a.a. $x$ (we slightly abuse the notation by denoting density of $\widetilde{\rho}_t$ with respect to the Lebesgue measure as $\widetilde{\rho}_t$). On the second set the Radon-Nikodym derivative does not exist $\widetilde{\rho}_t\otimes\lambda^1(t)$-a.e., but the lower derivative exists and is positive (the proof can be based for instance on  \cite[Section 1.6.1, Lemma 1]{EG-15}). 
Therefore, ultimately \eqref{main_eq1} holds for a.a. $t\in[0,T]$ and $\widetilde{\rho}_t$-a.e, which implies that $\sigma_{t,x} = \delta_{u(t,x)}$ for a.a. $t\in[0,T]$ and $\widetilde{\rho}_t$-a.e $x$, hence the proof for the non-atomic part with $\alpha>d$ is finished.

\medskip
\noindent
$\diamond$ {\sc Step 2.} {\it The non-atomic case with $\alpha = d$.}

\medskip
\noindent
The proof in the case of $\alpha=d$ follows similarly to Step 1. The first change we make is that we take
 $\beta_\eta = 1, $
which is trivially well defined. The convolution $K_\eta*\widetilde{\rho}_t$ is also well defined due to the uniform boundedness of the support of $\widetilde{\rho}_t$. Then the proof progresess similar to Step 1 until we reach \eqref{main1-p3}. This time around the elegant computation of the integrals is unavailable. Instead, we introduce a dyadic decomposition of the ball $B(x,1)$ setting $ B_k := B(x,2^{-k})$  for $k\in\{0,1,...\}$.
Then for all $x\neq x'\in B(x,1) = B_0$ we have

$$ |x-x'|^{-d}\geq  \sum_{k=1}^\infty 2^{(k-1)d}\chi_{B_k}(x'). $$

\noindent
Observe that for all sufficiently small $\eta>0$ and all $x'\in B_0$, there exists a maximal $k(\eta)\in{\mathbb N}$ such that

$$ (|x-x'|+\eta)^{-d}\gtrsim \sum_{i=1}^{k(\eta)} 2^{(i-1)d}\chi_{B_i}(x'), $$
with $k(\eta)\to\infty$ as $\eta\to 0$.
We integrate the above inequality obtaining
\begin{equation}\label{series}
K_\eta * \widetilde{\rho}_t\gtrsim \sum_{i=1}^{k(\eta)} 2^{id}\int_{B_i} \dd \widetilde{\rho}_t(x') \approx \sum_{i=1}^{k(\eta)} \frac{\widetilde{\rho}_t(B_i)}{\lambda^d(B_i)}.
\end{equation}
Now going back to \eqref{main1-p2}, where the right-hand side does not converge to $0$ but is only bounded, we again use Fatou's lemma obtaining
\begin{align*}
    \int_0^T \int_{\R^d} \left(\int_{\R^{d}} |v-u(t,x)|^{\alpha+2}\dd\sigma_{t,x}(v)\right) \liminf_{\eta\to 0}(K_\eta * \widetilde{\rho}_t(x))\dd \widetilde{\rho}_t(x)\dd t <\infty.
\end{align*}
The above is possible only if the integrand is finite $\rho_t\otimes\lambda^1(t)$-a.e. However, arguing using the  Radon-Nikodym derivative of measure $\widetilde{\rho}_t$ with respect to the Lebesgue measure, we infer that the summands in the series on the right-hand side of \eqref{series} converge (in the sense of limes inferior with respect to $i\to\infty$) to a positive limit for a.a. $t$ and $\rho_t$-a.a. $x$. Thus the series is divergent and
$$\liminf_{\eta\to 0}K_\eta*\widetilde{\rho}_t(x) = \infty$$
for a.a. $t$ and $\rho_t$-a.a. $x$, which leads to the conclusion that
$$ \int_{\R^{d}} |v-u(t,x)|^{\alpha+2}\dd\sigma_{t,x}(v) = 0$$
for a.a. $t$ and $\rho_t$-a.a. $x$. This finishes the second step.


\begin{rem}\rm
Before we proceed with the final step of the proof, where we deal with the atomic part of $\rho_t$, two comments are in order. First, the proof in Step 2 could successfully work for Step 1. The reason we kept the $\alpha>d$ case separately is simply that we find it more elegant. The second comment is related to the pivotal role of inequality \eqref{main1-crux}. If we were unable to ''add'' the diagonal set $\Delta$ in \eqref{main1-crux}, as compared to \eqref{main1-as1}, we would not be able to express ${\mathcal E}_\eta[\widetilde{\mu}_t]$ using the convolution with $K_\eta$. This problem is circumvented differently in \cite{JR-16}, where the authors use an additional cut-off function to separate the convolution from the diagonal set. Since they cut-off the diagonal set, ultimately, their approach leads to the same results, in that they need to consider the atomic part of $\rho_t$ separately. 
\end{rem}

\medskip
\noindent
$\diamond$ {\sc Step 3.} {\it The atomic case with $\alpha\geq 1$.}

\noindent
The idea of the proof in the atomic case is significantly different and relies not only on assumption \eqref{main1-as1} but also on the fact that $\{\mu_t\}_{t\in[0,T]}$ is locally mass preserving and steadily flowing (recall conditions (MP) and (SF) in the introduction). We proceed by contradiction. Assuming that there exists an atom whose velocity is not concentrated in a Dirac mass we can separate the velocity into two regions. Then, due to property (SF) the atom splits into two measures that are initially arbitrarily close to each other, i.e. close to the diagonal set $\Delta$ (but crucially -- outside of $\Delta$), which is singular in the integral in \eqref{main1-as1}. This leads to a blow-up of the integral in \eqref{main1-as1} and to the contradiction. The main challenge of the proof lies in proving that the splitting Dirac delta does not lose too much mass along the flow. This is where  properties (MP) and (SF) of $\{\mu_t\}_{t\in[0,T]}$ come into play. 


Throughout the proof let $A$ be the full measure set introduced in (SF) (intersected with the full measure set on which $\rho_t$ are well defined) and let $B(M)$ be the maximal velocity domain as set in Definition \ref{weakkinetic}. Aiming at contradiction we suppose that there is an atom at the time $t_0\in A$, position $x_0\in\R^d$ and mass $\rho_0(t_0)$, cf. \eqref{atomandnot}, such that
$\sigma_{t_0,x_0}\in {\mathcal P}(\R^d_v)$ is not a Dirac delta. Then, by a quick construction based on a smooth partition of unity, we can find three smooth, compactly supported functions $0\leq \phi_0,\phi_1,\phi_2\leq 1$, $\sum_{i=0}^2\phi_i\geq \chi_{B(M)}$, with the following properties
\begin{equation}\label{koniec2}
\begin{split}
    \int_{\R^d}\phi_i(v)\dd \sigma_{t_0,x_0}(v) &=: m_i >0 \quad\mbox{for }\ i=1,2,\\
     \int_{\R^d}\phi_0(v)\dd\sigma_{t_0,x_0}(v) &= 1 - m_1- m_2, 
     \\
    {\rm dist}({\rm spt}(\phi_1),{\rm spt}(\phi_2)) &>\delta>0, \quad\mbox{for some }\ \delta>0.
\end{split}
\end{equation}
Recalling the notation from (SF) and \eqref{atomandnot} we note that the first two equations above imply that
\begin{align*}
    \rho_{t_0,t_0}[\phi_i](\{x_0\}) &= \rho_0(t_0) m_i,\qquad i\in\{1,2\},\\ 
    \rho_{t_0,t_0}[\phi_0](\{x_0\}) &= \rho_0(t_0)(1 -  m_1 -  m_2). 
\end{align*}
Denote $m:= \min\{m_1,m_2\}$. 
 The steady flow (SF) of $\mu$ implies that the ${\mathcal M}_+(\R^d)$-valued functions $A\ni t\mapsto \rho_{t_0,t}[\phi_i]$ are narrowly continuous at $t=t_0$. Since the singleton $\{x_0\}$ is a closed set, this translates to the upper-semicontinuity of the functions $A\ni t\mapsto \rho_{t_0,t}[\phi_i](\{x_0\})$ at $t=t_0$. Thus there exists a time interval $[t_0,t_*)$ such that for all $t\in [t_0,t_*)\cap A$ we have
\begin{equation}\label{main1-pst2-1}
\begin{split}
        \rho_{t_0,t}[\phi_i](\{x_0\}) &\leq \rho_0(t_0)\left(m_i + \frac{m}{6}\right),\qquad i\in\{1,2\},\\ 
    \rho_{t_0,t}[\phi_0](\{x_0\}) &\leq \rho_0(t_0)\left(1 -  m_1 -  m_2 + \frac{m}{6}\right).
\end{split}
\end{equation}
Moreover, since $\mu$ is mass-preserving, we have

\begin{align*}
    \sum_{i=0}^2\rho_{t_0,t}[\phi_i](\{x_0\}) \geq \int_{\{x_0\}\times \overbar{B(M)}} \dd T^{t_0,t}_\#\mu_t = \int_{(T^{t_0,t})^{-1}(\{x_0\}\times \overbar{B(M)})} \dd \mu_t\\
    = \int_{(\{x_0\}+(t-t_0)\overbar{B(M)})\times \overbar{B(M)}} \dd \mu_t = \rho_t(\{x_0\}+(t-t_0)\overbar{B(M)}))\geq \rho_{t_0}(\{x_0\}) =  \rho_0(t_0)
\end{align*}

\noindent which together with $\eqref{main1-pst2-1}$ implies that
\begin{equation}\label{main1-pst2-2}
    \begin{split}
    \rho_{t_0,t}[\phi_i](\{x_0\}) &\geq \rho_0(t_0)\left( m_i - \frac{m}{3}\right) \geq \rho_0(t_0)\frac{m}{2},\qquad i\in\{1,2\}.
    \end{split}
\end{equation}

Let
$$\Delta_t = \{(x,v,x',v): (x-x') + (t-t_0)(v-v')= 0\}.$$

\noindent
Now recalling assumption \eqref{main1-as1} we use the pushforward change of variables formula \eqref{changeofvariables} to write
\begin{align*}
   \infty> \int_0^T\int_{\R^{4d}\setminus\Delta}\frac{|v-v'|^{\alpha+2}}{|x-x'|^\alpha}\dd \left[\mu_t\otimes\mu_t\right]\dd t\\
   = \int_0^T\int_{\R^{4d}\setminus \Delta_t}\frac{|v-v'|^{\alpha+2}}{|(x-x') + (t-t_0)(v-v')|^\alpha}\dd \left[T^{t_0,t}_\#\mu_t\otimes T^{t_0,t}_\#\mu_t\right]\dd t\\
    \geq \int_{(t_0,t_*)\cap A}\int_{(\{x_0\}\times B(M))^2\setminus \Delta_t}\frac{|v-v'|^{\alpha+2}}{|(x-x') + (t-t_0)(v-v')|^\alpha}\phi_1(v)\phi_2(v')\dd \left[T^{t_0,t}_\#\mu_t\otimes T^{t_0,t}_\#\mu_t\right]\dd t
\end{align*}

\noindent Observe that, by the third property in \eqref{koniec2}, on $((t_0,t_*)\cap A)\times(\{x_0\}\times B(M))^2$ intersected with the supports of $\phi_1$ and $\phi_2$, we have 
$$0<(t-t_0)\delta<|(x-x')+(t-t_0)(v-v')|\leq 2(t-t_0)M,$$ 
which in particular ensures that in the domain of integration above we have 
$$(\{x_0\}\times B(M))^2\setminus \Delta_t=(\{x_0\}\times B(M))^2.$$
Combining the above inequality with the third property in \eqref{koniec2} we obtain
\begin{align*}
   \infty> \delta^{\alpha+2}\int_{(t_0,t_*)\cap A} \frac{1}{(2(t-t_0)M)^\alpha}  \int_{(\{x_0\}\times B(M))^2} \phi_1(v)\phi_2(v')\dd \left[T^{t_0,t}_\#\mu_t\otimes T^{t_0,t}_\#\mu_t\right]\dd t\\
   = \delta^{\alpha+2}\int_{t_0}^{t_*} \frac{1}{(2(t-t_0)M)^\alpha} \left(\int_{\{x_0\}\times \R^d}\phi_1(v) \dd T^{t_0,t}_\#\mu_t\right)\left( \int_{\{x_0\}\times \R^d}\phi_2(v) \dd T^{t_0,t}_\#\mu_t\right) \dd t\\
   = \delta^{\alpha+2}\int_{t_0}^{t_*} \frac{1}{(2(t-t_0)M)^\alpha} \rho_{t_0,t}[\phi_1](\{x_0\}) \rho_{t_0,t}[\phi_2](\{x_0\}) \dd t\\
   \stackrel{\eqref{main1-pst2-2}}{\geq} \delta^{\alpha+2}(\rho_0(t_0))^2\left(\frac{m}{2}\right)^2\int_{t_0}^{t_*} \frac{1}{(2(t-t_0)M)^\alpha}\dd t = \infty.
\end{align*}
The contradiction indicates that the initial assumption that there exists $t_0\in A$ in which there is a non-monokinetic atom was false. Thus combining all of the information obtained in all of the steps we conclude that for a.a. $t\in[0,T]$ and $\rho_t$-a.a. $x$ the measure $\sigma_{t,x}$ is concentrated in a Dirac delta, and thus $\mu$ is monokinetic for a.a. $t\in[0,T]$ and $\rho_t$-a.a. $x$. 
\end{proof}

\begin{rem}\rm
Observe that the only assumption on $\alpha$ necessary in Step 3 of the above proof is that $\alpha \geq 1$ so that the function $t\mapsto (t-t_0)^{-\alpha}$ is not integrable near $t=t_0$.
\end{rem}

We end this section by sketching the proof of Theorem \ref{main2}. In essence, it amounts to showing that solutions in the sense of Definition \ref{weakkinetic} satisfy assumptions of Theorem \ref{main1}.

\begin{proof}[Proof of Theorem \ref{main2}]
    Since $\mu$ is narrowly continuous, it is easy to show that it is steadily flowing. Then the remaining assumptions of Theorem \ref{main1}, except (MP), follow from item (i) in Definition \ref{weakkinetic} and from Proposition \ref{relaxdef}. We omit the details. The only nontrivial assumption to check is that $\mu$ is locally mass preserving. In the proof of Theorem \ref{main1} we only use condition (MP) in the special case, when $C$ is a singleton. Thus, we shall only prove it in such a case. We do it in the appendix, Proposition \ref{solMP}.
    \end{proof}

\section{Compactness of atomic solutions}\label{sec:compact}

The main purpose of this section is to prove compactness  of the sequence $\left\{\mu^N \right\}_{N=1}^\infty$ of atomic solutions (cf. Definition \ref{empirical}). The goal is to extract a convergent subsequence leading to the mean-field limit. However, as explained in Section \ref{intro:mf}, it is unclear whether it is possible to obtain compactness sufficient for a kinetic mean-field limit. Thus, we settle for a weaker property, which is satisfactory for the purpose of a micro- to macroscopic mean-field limit.
The results of this section introduce the necessity to assume that $\alpha\leq 2$, which together with the previous section restricts the range of admissible  $\alpha$ to $[d,2]$.


\begin{prop}\label{compactness}
Fix $T>0$ and let $\{\mu^N\}_{N=1}^\infty$ be any sequence of atomic solutions (cf. Definition \ref{empirical}), uniformly compactly supported in $(T+1)B(M)\times B(M)$ and $\rho^N$ -- their local density defined in \eqref{localquant}.
Then the following assertions hold.
\begin{enumerate}[label=(\roman*)]
    \item The sequence $\{\mu^N\}_{N=1}^\infty$ is narrowly relatively compact in ${\mathcal M}_+([0,T]\times\R^{2d})$ and any of its limits has the 1D Lebesgue measure as its $t$-marginal. 
    \item The sequence $\{\rho^N\}_{N=1}^\infty$ treated as measure-valued functions on $[0,T]$ is relatively compact in $${C}\left( [0,T]; (\mathcal{P}(\R^d), d_{BL})\right).$$
\end{enumerate}
\end{prop}
\begin{proof}
 Narrow relative compactness of $\{\mu^N\}_{N=1}^\infty$ in ${\mathcal M}_+([0,T]\times\R^{2d})$ follows from Banach-Alaoglu theorem due to the uniform compactness of supports of $\mu^N$ and the fact that they are strongly bounded in the $TV$ topology by the constant $T$, see Remark \ref{rem:compact}. In order to prove that any limit $\mu$ of elements in $\{\mu^N\}_{N=1}^\infty$ disintegrates as
 $$ \mu(t,x,v) = \mu_t(x,v)\otimes \lambda^1(t) $$
 we proceed as follows.
 We know, by the narrow convergence $\mu^N\rightharpoonup\mu$, that for all bounded continuous functions $g=g(t,x,v)$ we have 
$$
\int_0^T \int_{\R^{2d}} g(t,x,v) \dd\mu_t^N \dd t\xrightarrow{N \rightarrow \infty} \int_0^T \int_{\R^{2d}} g(t,x,v) \dd \mu.
$$
Let $\pi_t$ be the projection of $(t,x,v)$ onto the time coordinate and $\nu:= (\pi_t)_{\#}\mu$. Then due to the disintegration theorem, we have
\begin{equation}\label{nudef}
    \mu(t,x,v) = \mu_t(x,v) \otimes \nu(t)
\end{equation}
for some probabilistic family $\{\mu_t\}_{t\in[0,T]}$.
It remains to  prove that $\nu = \lambda^1$. To this end, we test the narrow convergence with $g(t,x,v) = \phi(\pi_t(t,x,v)) = \phi(t)$, with an arbitrary continuous function $\phi$, obtaining
\begin{align*}
\int_0^T  \phi(t) \dd t &= \int_0^T \int_{\R^{2d}} g(t,x,v) \dd \mu_t^N \dd t \xrightarrow{N\rightarrow \infty} \int_0^T \int_{\R^{2d}} g(t,x,v) \dd \mu\\
&= \int_0^T  \phi(t) \int_{\R^{2d}}\dd \mu_t(x,v) \otimes \nu(t) = \int_0^T \phi(t) \dd\nu(t).
\end{align*}
Since the last and the first term are independent of $N$ and $\phi$ is arbitrary, we conclude that $\nu(t) = \lambda^1(t)$.
 
 \medskip
 Item (ii) follows immediately from  Arzela-Ascoli theorem, which requires us to prove that $\{\rho^N\}_{N=1}^\infty$ are equicontinuous in the bounded-Lipschitz metric, uniformly bounded and pointwise relatively compact. 
 Once again boundedness and pointwise relative compactness follow from the fact that $\mu_t^N$ (and hence $\rho_t^N$) is probabilistic and compactly supported, as in the proof of item (i).
 
 The remainder of the proof is dedicated to equicontinuity of $\{\rho^N\}_{N=1}^\infty$. Fix $0 < t_1 < t_2 < T$ and $\varepsilon < \min\{ t_1, T-t_2\}$ and let
 
\begin{equation}\label{cutof}
    \psi_{t_1, t_2, \varepsilon}(t) := 
\begin{cases}
0, \quad t\le t_1-\varepsilon \text { or } t\ge t_2 + \varepsilon,\\
1, \quad t_1+\varepsilon \le t \le t_2 - \varepsilon,\\
\text{linear on } |t-t_i|\le \varepsilon, \; i=1,2.
\end{cases}
\end{equation}
\noindent
Furthermore let $\phi=\phi(x)$ be a $C^1$ function with $|\phi|_\infty \le 1$, $[\phi]_{Lip}\le 1$ (thus, it is an admissible test function in the definition of $d_{BL}$). We test the weak formulation \eqref{weakkinetic} for $\mu_t^N$ with $\psi_\epsilon(t)\phi(x)$ (droping $t_1$ and $t_2$ in the subscript of $\psi$ for the sake of clarity) obtaining
\begin{align*}
    0 = \underbrace{\int_{\R^{2d}} \psi_\epsilon(0)\phi(x) \dd \mu_0^N}_{=0} &+
    \int_0^T \int_{\R^{2d}} \partial_t\left( \psi_\epsilon(t)  \phi (x)\right) \dd \mu_t^N \dd t + 
    \int_0^T \int_{\R^{2d}} \nabla_x \left(\psi_\epsilon(t) \phi(x)\right)\cdot v \dd \mu_t^N \dd t \\
     &= \sum\limits_{i=1}^2 \int_{t_i - \varepsilon}^{t_i+\varepsilon} \partial_t \psi_\epsilon(t) \int_{\R^{2d}}\phi(x) \dd \mu_t^N \dd t + \int_{t_1 -\varepsilon}^{t_2 + \varepsilon}\psi_\epsilon(t)  \int_{\R^{2d}} \nabla_x \phi(x)\cdot \,v \dd \mu_t^N \dd t\\
   &= \sum\limits_{i=1}^2 \int_{t_i - \varepsilon}^{t_i+\varepsilon} \partial_t \psi_\epsilon(t) \int_{\R^d}\phi(x) \dd \rho_t^N \dd t + \int_{t_1 -\varepsilon}^{t_2 + \varepsilon}\psi_\epsilon(t)  \int_{\R^{2d}} \nabla_x \phi(x)\cdot \,v \dd \mu_t^N \dd t\\
     &=: \sum\limits_{i=1}^2 I_{1,i} + I_2.
\end{align*}
From the definition of $\psi$ we infer that $\partial_t \psi_\epsilon =(2\varepsilon)^{-1}$ on $(t_1 - \varepsilon, t_1+\varepsilon)$ and $-(2\varepsilon)^{-1}$ on $(t_2 - \varepsilon, t_2+\varepsilon)$. Note that the function $t\mapsto \int_{\R^d} \phi \dd \rho_t^N$ is continuous, since $\mu^N_t$ is narrowly continuous, and $\phi$ is locally bounded and continuous. 
Thus using the fundamental theorem of calculus, we obtain
\begin{equation}\label{I12}
I_{1,1} + I_{1,2} = \frac{1}{2\varepsilon}  \int_{t_1 - \varepsilon}^{t_1+\varepsilon} \int_{\R^d} \phi\,  d\rho_t^N \dd t - \frac{1}{2\varepsilon}  \int_{t_2 - \varepsilon}^{t_2+\varepsilon} \int_{\R^d} \phi\,  d\rho_t^N \dd t \xrightarrow{\varepsilon\rightarrow 0} 
\int_{\R^d} \phi \dd \rho_{t_1}^N - \int_{\R^d} \phi \dd \rho_{t_2}^N
\end{equation}
for all $t_1<t_2\in(0,T)$.
On the other hand, since $\mu^N_t$ are probabilistic with supports bounded in terms of $M$, and $|\nabla_x\phi|_\infty\leq 1$, we have
\begin{align}\label{I2}
|I_2|&\le \int_{t_1 -\varepsilon}^{t_2 + \varepsilon} \int_{\R^{2d}} |v| |\nabla_x \phi| \dd \mu_t^N \dd t \le
M \left(t_2 - t_1 + 2 \varepsilon \right)
\leq 2M \left(t_2 - t_1\right)
\end{align}
for sufficiently small $\varepsilon>0$.
Combining \eqref{I12} with \eqref{I2}, we find
\begin{align*}
\left|\int_{\R^d} \phi \dd \rho_{t_1}^N - \int_{\R^d} \phi \dd \rho_{t_2}^N \right| &\le 2M  |t_1-t_2|.
\end{align*}
Taking the supremum over all such admissible functions $\phi$ yields
\begin{align*}
  d_{BL}(\rho_{t_1}^N, \rho_{t_2}^N) &\le 2M |t_1-t_2|
\end{align*}
 for all $t_1<t_2\in(0,T)$. Thus the family of measure-valued functions $\rho^N$ is uniformly Lipschitz continuous in $(0,T)$ with values in $({\mathcal P}(\R^d),d_{BL})$. Each of these functions is continuous on the whole $[0,T]$, which ensures that the uniform Lipschitz continuity holds on the whole $[0,T]$ too. This yields the required equicontinuity and the proof is finished.
\end{proof}

Finally, since the interaction term in \eqref{weakeqkinetic} contains integration with respect to the measure $[\mu_t \otimes \mu_t] \otimes \lambda^1$, we need to be able to pass to the limit with terms of this type as well. Even though such terms seem nonlinear and the framework is not regular,  somewhat surprisingly, the following proposition holds true.

\begin{prop}\label{lem:mutimesmu}
Under the assumptions of Proposition \ref{compactness} let $\mu^{N_k}$ be a subsequence convergent to $\mu\in {\mathcal M}_+([0,T]\times\R^{2d})$.
Then, up to a subsequence, $[\mu^{N_k}_t\otimes\mu^{N_k}_t]\otimes\lambda^1(t)$ converges narrowly to some $\lambda\in{\mathcal M}_+([0,T]\times\R^{4d})$, which can be disintegrated as 
$$
[\mu_t^{N_k} (x,v) \otimes \mu_t^{N_k} (x',v')] \otimes dt \xrightharpoonup{k \rightarrow \infty} \lambda(t,x,v,x',v') =  [\mu_t(x,v) \otimes \mu_t(x', v')] \otimes \lambda^1(t).
$$
\end{prop}

\begin{proof}
Narrow relative compactness of  $\{[\mu_t^{N_k} \otimes \mu_t^{N_k}] \otimes \lambda^1\}_{k\in{\mathbb N}}$ follows from Proposition \ref{compactness}, and therefore there exists a subsequence (not relabeled) and a measure $\lambda \in \mathcal{M}_+([0,T]\times \R^{4d})$ such that 
$$
[\mu_t^{N_k}\otimes \mu_t^{N_k}] \otimes \lambda^1(t) \xrightharpoonup{k \rightarrow \infty} \lambda.
$$
Then, again by Proposition \ref{compactness}, disintegrating $\lambda$ with respect to $t$, we obtain 
$$
\lambda(t,x,v,x',v') = \lambda_t (x,v,x',v') \otimes \lambda^1(t).
$$
Therefore the only nontrivial part is to prove that for a.a. $t\in[0,T]$ we have 
\begin{equation}\label{nielin}
    \lambda_t(x,v,x',v') = \mu_t(x,v) \otimes \mu_t(x',v').
\end{equation}
Denote the uniform superset of all supports of $\mu^{N_k}_t$ and $\mu_t$ as $\Omega \subset\subset (T+1)B(M)\times B(M)$. We shall prove \eqref{nielin} by integrating against arbitrary functions in $C^2(\Omega)$, or rather, in its countable dense subset ${\mathcal D}$. Fix any 
$$
\phi \in {\mathcal D} \subset C^2_c(\Omega),\qquad \|\phi\|_{C^2}\leq 1.
$$
Consider the integral
$$ \int_{\R^{2d}}\phi\dd\mu_t^{N_k} = \frac{1}{N}\sum_{i=1}^N\phi(x_i^{N_k}(t),v_i^{N_k}(t)), $$
where $(x_i^{N_k},v_i^{N_k})_{i=1}^{N_k}$ is the solution of the particle system \eqref{micro} corresponding to $\mu^{N_k}$. Since each such solution is smooth, the above integrals are differentiable with respect to $t$.
Following the standard calculations exploiting the usual symmetrization trick, we obtain 
\begin{align*}
   \left| \frac{d}{dt}\int_{\R^{2d}}\phi\dd\mu_t^{N_k} \right| 
   &\le \frac{1}{N_k}\sum\limits_{i=1}^{N_k} \left| \nabla_x \phi \right| \left| v_i^{N_k}\right| + 
   \frac{1}{2N_k^2}\sum\limits_{i\neq j=1}^{N_k} \frac{\left|\nabla_v\phi(x_i^{N_k}, v_i^{N_k}) - \nabla_v\phi(x_j^{N_k}, v_j^{N_k})\right| |v_j^{N_k}-v_i^{N_k}|}{|x_i^{N_k} - x_j^{N_k}|^\alpha} \\
   &\le M\|\nabla_x \phi\|_\infty + \frac{1}{2N_k^2}[\nabla_v \phi]_{Lip} \sum\limits_{i \neq j=1}^{N_k} |x_i^{N_k} - x_j^{N_k}|^{2-\alpha}+ \frac{1}{2N_k^2}[\nabla_v \phi]_{Lip} \sum\limits_{i \neq j=1}^{N_k} \frac{|v_i^{N_k}-v_j^{N_k}|^2}{|x_i^{N_k}-x_j^{N_k}|^{\alpha}} \\
   &\le M + \frac{1}{2}(2(T+1)M)^{2-\alpha} + \frac{1}{2}D[\mu^{N_k}_t],
\end{align*}
where the last line follows from the fact that $\|\phi\|_{C^2}\leq 1$ and from uniform compactness of supports of $\mu^N$. Integrating with respect to time in $[0,T]$ and using Proposition \ref{flock} we find
$$ \left\|\frac{d}{dt}\int_{\R^{2d}}\phi\dd\mu^{N_k}_t\right\|_{L^1([0,T])} \leq TM + \frac{T}{2}(2(T+1)M)^{2-\alpha} + \frac{1}{2}E[\mu^{N_k}_0]\leq TM + \frac{T}{2}(2(T+1)M)^{2-\alpha}+\frac{M^2}{2}.$$
Thus the family of functions $t\mapsto \int_{\R^{2d}}\phi\dd\mu^{N_k}_t$ is bounded in $W^{1,1}([0,T])$, and thus, it is compactly embedded in $L^2([0,T])$. Passing to a subsequence, for our fixed $\phi$, we have the convergence
$$
\int_{\R^{2d}}\phi \dd\mu^{N_{k_1}}_t\longrightarrow \psi_\phi(t)\quad \text{in}\quad L^2([0,T])
$$
where $\psi_\phi$ is some $\phi$-dependent limit. Using a diagonal argument we construct the subsequence $N_{k_\infty}$, such that
$$
\forall \phi \in {\mathcal D},\ \|\phi\|_{C^2}\leq 1 \quad\mbox{we have}\quad \int_{\R^{2d}}\phi \dd\mu^{N_{k_\infty}}_t \longrightarrow \psi_\phi(t) \text{ in }L^2([0,T]).
$$
Here we come back to the narrow convergence of the original series $\mu^{N_k}\rightharpoonup\mu$. For $\psi \in C([0,T])$, we have
$$
\int_0^T \psi(t)\int_{\R^{2d}} \phi\dd\mu^{N_{k_\infty}}_t \dd t = \int_0^T\int_{\R^{2d}} \psi(t) \phi(x,v) \dd \mu_t^{N_{k_\infty}} \dd t \xrightarrow{N_{k_\infty} \rightarrow \infty}
\int_0^T \psi(t) \int_{\R^{2d}} \phi \dd \mu_t \dd t.
$$
Since $\psi$ belongs to a dense subspace $C([0,T])$ of $L^2([0,T])$, by uniqueness of the weak limit, we deduce that in fact 
$$
\psi_\phi(t) = \int_{\R^{2d}}\phi\dd\mu_t\quad \mbox{as elements of }\ L^2([0,T])\ \mbox{ for all } \phi\in{\mathcal D}.
$$

Fix two functions $\phi_1, \phi_2\in {\mathcal D}$ with $\|\phi_i\|_{C^2}\leq 1$. The sequences of functions $t\mapsto \int_{\R^d}\phi_i\dd\mu^{N_{k_\infty}}_t$, $i\in\{1,2\}$, are strongly convergent in $L^2([0,T])$, and hence their product is strongly (and weakly) convergent in $L^1([0,T])$. Since $C([0,T]) \subset L^\infty([0,T]) = L^1([0,T])^*$, we infer that for all $\psi\in C([0,T])$ we have

\medskip
\begin{align*}
\int_0^T \int_{\R^{4d}} \psi(t) \phi_1(x,v) \phi_2(x',v') \dd \left[\mu_t^{N_{k_\infty}}(x,v)\otimes \mu_t^{N_{k_\infty}}(x',v')\right] \dd t =
\int_0^T \psi(t) \int_{\R^{2d}}\phi_1\dd\mu^{N_{k_\infty}}_t \int_{\R^{2d}}\phi_2\dd \mu^{N_{k_\infty}}_t\dd t \\
\rightarrow \int_0^T \psi(t) \int_{\R^{2d}}\phi_1\dd\mu_t \int_{\R^{2d}}\phi_2\dd \mu_t\dd t = 
\int_0^T \int_{\R^{4d}} \psi(t) \phi_1(x,v) \phi_2(x',v') \dd\left[\mu_t(x,v)\otimes \mu_t(x',v')\right]\dd t.
\end{align*}

\medskip
By a standard density argument, the above convergence extends to all $\phi_1, \phi_2$ in $C_c(\Omega)$. We need this convergence for arbitrary function $g(t,x,v,x',v')\in C([0,T]\times\Omega^2)$. By Stone-Weierstrass theorem, finite linear combinations of products $\psi(t) \phi_1(x,v) \phi_2(x',v')$ are dense in $C([0,T]\times\Omega^2)$, so we can further extend our result. Finally we arrive at the conclusion that
$$
\int_0^T \int_{\R^{4d}} g(t,x,v,x',v') \dd \left[\mu_t^{N_{k_\infty}}(x,v)\otimes \mu_t^{N_{k_\infty}}(x',v')\right] \dd t \rightarrow \int_0^T \int_{\R^{4d}} g(t,x,v,x',v') \dd [\mu_t(x,v)\otimes \mu_t(x',v')] \dd t 
$$
for all $g\in C([0,T]\times\Omega^2)$. Due to the uniqueness of the narrow limit we conclude that
$$ \lambda = [\mu_t(x,v)\otimes \mu_t(x',v')]\otimes \lambda^1(t)  $$
and the proof is finished.
\end{proof}

\section{Micro- to macroscopic mean-field limit}\label{sec:mf}
The main goal of this section is to prove the last main result Theorem \ref{main3}. The first issue is to prepare the initial data for the mean-field approximation.
Following assumptions of Theorem \ref{main3}, fix $\rho_0 \in \mathcal{P}(\R^d)$ with bounded support compactly contained in the ball $B(M)$ (where $M$ is sufficiently large) and $u_0\in L^\infty(\rho_0)$ with $\|u_0\|_{L^\infty(\rho_0)}\leq M$. 
Our first goal is to approximate $\rho_0$ (and $u_0$) by a sequence of empirical measures. On the one hand such approximation results are well known and boil down to the law of large numbers. On the other hand our setting is non-standard. In particular the function $u_0$ is defined $\rho_0$-a.e., which makes it unclear if we can evaluate it in points emerging in the empirical measures. We perform the approximation in Lemma \ref{approx} in the appendix ensuring the existence of 
\begin{equation}\label{approxini}
    \begin{aligned}
    \rho_0^N(x) &:= \frac{1}{N} \sum\limits_{i=1}^N\delta_{x_{i0}^N}(x) \xrightharpoonup{N \rightarrow \infty} \rho_0, & x_{i0}^N\neq x_{j0}^N\in\R^d \text{ for }i\neq j,\\
m_0^N(x) &:= \frac{1}{N} \sum_{i=1}^N v_{i0}^N\delta_{x_{i0}^N}(x) \xrightharpoonup{N \rightarrow \infty} u_0\rho_0, & v_{i0}^N\in\R^d,
\end{aligned}
\end{equation}
where $m_0^N$ and $u_0\rho_0$ are understood as vector measures in ${\mathcal M}(\R^d)$. Then 
$\left(x_{i0}^N, v_{i0}^N \right)_{i=1}^N \in \R^{2dN}$ serve as good, non-collisional initial data for the particle system \eqref{micro}. Let $(x_i^N(t),v_i^N(t))_{i=1}^N$ be the corresponding  unique, smooth, non-collisional solution granted by Theorem \ref{exist_micro}.
Following Definition \ref{empirical}, let $\mu^N \in \mathcal{M}\left( [0,T] \times \R^{2d} \right)$ be the associated atomic solution. Based on the properties of solutions to the particle system (see Proposition \ref{flock} and \ref{eneq}), both compactness results from Section \ref{sec:compact} apply to  $\{\mu^N\}_{N=1}^\infty$. Additionally, the following lemma holds true. 

\begin{lem}\label{lem:zbiegze}
    Suppose $\mu^N$, $\mu$ are given as above and $\mu^N$ are uniformly (with respect to $N$) compactly supported, i.e. there exists $M>0$ such that for each $N$ and $t\in[0,T]$
    $$
    {\rm spt}(\mu_t^N) \subset\subset (T+1)B(M)\times B(M).
    $$
    Then $\mu$ is uniformly compactly supported in $(T+1)B(M)\times B(M)$ and there exists a subsequence $\{ \mu_t^{N_k}\}_{k=1}^\infty$, such that
        $$
        E[\mu_t^{N_k}] \xrightarrow{ k \rightarrow \infty} E[\mu_t] \text{ for all }t\in A_E,
        $$
        where $E$ is the kinetic energy introduced in Section \ref{sec:micro} and $A_E$ is some full measure subset of $[0,T]$.
\end{lem}
\begin{proof}
   The desired inclusion of ${\rm spt}(\mu^N_t)$ follows from the properties of narrow convergence by testing with an admissible function localised outside the set $(T+1)B(M)\times B(M)$. The proof of convergence of the energy is essentially the same as in the first part of Proposition \ref{lem:mutimesmu}, where we established that
   $$
   \int_{\R^{2d}}\phi \dd \mu_t^{N} \rightarrow \int_{\R^{2d}}\phi \dd \mu_t \quad\mbox{ in }\ L^2([0,T])\ \mbox{ and thus also for a.a. }\ t\in[0,T]
   $$
   (up to a subsequence) for all $\phi(x,v) \in C^2_c$. Recalling the definition of $E[\mu_t^N]$ and taking $\phi=|v|^2$ (which is an admissible function due to uniform compactness of all measures), we conclude the proof.
   
\end{proof}

Based on the above lemma and results from Section \ref{sec:compact}, we construct our approximative solutions and their limit in accordance with the following definition.

\begin{definition}[\textbf{Approximative solutions}]\label{def:MFL}
Let $T>0$ be fixed. For a given bounded, compactly supported $\rho_0\in \mathcal{P}(\R^{d})$ and $u_0 \in L^\infty(\rho_0)$, $u_0: \R^d_x \mapsto \R^d$,  we define initial approximation of the form \eqref{approxini} (see Lemma \ref{approx} in the appendix for the construction). Moreover for each $N\in{\mathbb N}$ let $(x_i^N(t), v_i^N(t))_{i=1}^N$ and $\{\mu^N\}_{N=1}^\infty$ be the  solution of the particle system \eqref{micro} subject to initial data $(x_{i0}^N, v_{i0}^N)_{i=1}^N$ and the corresponding atomic solution, respectively.
Using Lemma \ref{lem:zbiegze} and the results of Section \ref{sec:compact}, we extract (without relabelling) a subsequence of $\left\{ \mu^N \right\}_{N=1}^\infty$ satisfying all of the following properties:
\begin{align*}
     \mu^N(t,x,v)&\xrightharpoonup{N \rightarrow \infty} \mu(t,x,v)= \mu_t(x,v) \otimes \lambda^1(t); \\
     \left[ \mu_t^N(x,v) \otimes \mu_t^N(x',v')\right] \otimes \lambda^1(t) &\xrightharpoonup{N \rightarrow \infty} \left[ \mu_t(x,v)\otimes \mu_t(x',v')\right]\otimes \lambda^1(t); \\
     \int_{\R^d_v} \dd \mu_t^N(x,v) =: \rho^N_t(x) &\xrightarrow{N \rightarrow \infty} \rho_t(x) \text{ in }C\left([0,T]; (\mathcal{M}_+(\R^d), d_{BL})\right);\\
     E[\mu_t^N] &\xrightarrow{N\rightarrow\infty} E[\mu_t] \text{ on } A_E,
\end{align*}
where $A_E$ is defined in Lemma \ref{lem:zbiegze}.
Moreover, $\mu_t$ is compactly supported for each $t\in[0,T]$, with bounds on support inherited from $\mu_t^N$. 
\end{definition}

\begin{rem}\label{rem:mu=sigmarho}\rm
    With the notation as above, for a.a. $t\in[0,T]$ the measure $\mu$ disintegrates as in \eqref{mudis}.  More importantly the $x$-marginal of $\mu_t$ defined as the local density $\rho_t$ in \eqref{localquant}, coincides for a.a. $t\in[0,T]$ with the limit of $\rho^N_t$. The proof exploits both the narrow convergence $\mu^N\rightharpoonup\mu$ and the convergence of $t\mapsto\rho_t^N$, and follows similarly to the proof that $\nu=\lambda^1$ in \eqref{nudef}, Proposition \ref{compactness}. We omit the details. 
\end{rem}

\subsection{Monokineticity of the limit family $\left\{ \mu_t \right\}_{t\in [0,T]}$}

Next goal in this section is to establish that the family $\left\{ \mu_t \right\}_{t\in [0,T]}$ from Definition \ref{def:MFL} fits into the framework of Theorem \ref{main1} and is therefore monokinetic. 

\begin{prop}\label{prop:m-p}
The family of measures $\{\mu_t\}_{t\in[0,T]}$ from Definition \ref{def:MFL} is locally mass preserving.
\end{prop}
\begin{proof}
First we shall prove that for each $N \in \mathbb{N}$ the family $\{\mu_t^N\}_{t \in [0,T]}$ is locally mass preserving, and then show that this property transfers to the narrow limit $\{\mu_t\}_{t\in[0,T]}$. Throughout the proof let $t \ge t_0\ge 0$ as well as $C\subset \R^d$ be fixed (as denoted in condition (MP)).  

\medskip
\noindent
$\diamond$ {\sc Step 1.} {\it Proof for $\mu^N$.}

Let $N\in \mathbb{N}$ be fixed and consider measures $\mu_t^N$, $\mu_{t_0}^N \in \mathcal{P}(\R^{2d})$ (note that the atomic solutions are defined for all $t$) and their local densities $\rho_{t}^N, \rho_{t_0}^N \in\mathcal{P}(\R^{2d})$, see \eqref{localquant}.
 We need to show that 
 \begin{equation}\label{cel_mpf_1}
 \rho_t^N(\overbar{C}+(t-t_0)\overbar{B(M)}) \ge \rho_{t_0}^N(C).
 \end{equation}
Since the atomic solutions are derived from the solutions to the particle system, the above inequality can be described as follows. Any particles, initially at $t=t_0$, situated inside $C$, are always (for all $t\geq t_0$) contained in the set $\overbar{C} + (t-t_0)\overbar{B(M)}$. This is clear, because the particles follow classical trajectories and move with speeds not exceeding $M$. 

\medskip
\noindent
$\diamond$ {\sc Step 2.} {\it Proof for $\mu_t$.}
By Definition \ref{def:MFL} we know that 
$$
\rho_t^N \rightharpoonup \rho_t, \quad \rho_{t_0}^N \rightharpoonup \rho_{t_0}.
$$
Let $B\left(1/n\right)$ be an open ball centered at $0$ of radius $\frac{1}{n}$ so that the Minkowski sum $C+ B(1/n)$ is open. Using Step 1 and Portmanteau theorem we find

\begin{align*}
    \rho_{t_0}(C) &\le \rho_{t_0}\Big(C+B(1/n)\Big)  
    \le \liminf_{N \rightarrow \infty} \rho_{t_0}^N \Big(C+B(1/n)\Big)
    \le \limsup_{N \rightarrow \infty} \rho_{t_0}^N \Big(\overbar{C+B(1/n)}\Big) \\
    &\stackrel{\eqref{cel_mpf_1}}{\le} \limsup_{N \rightarrow \infty} \rho_{t}^N   \Big(\overbar{C+B(1/n)+(t-t_0)B(M)}\Big)
    \le  \rho_{t}  \Big(\overbar{C+B(1/n)+(t-t_0)B(M)}\Big)\\
   & = \rho_{t}  \Big(\overbar{C+B(1/n)}+(t-t_0)\overbar{B(M)}\Big).
\end{align*}
The final equality above follows from the fact that the Minkowski sum of a closed set and a compact set is closed. It holds for all $n\in{\mathbb N}$ and the decreasing family of concetric balls $B(1/n)$.
Thus, it holds also for their finite intersections i.e. for any $j \in \mathbb{N}$ we have
$$
\rho_{t_0}(C) \le \rho_t \Big( \bigcap\limits_{n=1}^j\Big(\overbar{ C+B(1/n)}+(t-t_0)\overbar{B(M)}\Big) \Big).
$$
The family of sets on the right-hand side above is decreasing with respect to $j$, and thus
passing with $j\rightarrow \infty$ we find 
\begin{align*}
\rho_{t_0}(C) &\le \lim_{j\rightarrow \infty} \rho_t \Big( \bigcap\limits_{n=1}^j\Big(\overbar{ C+B(1/n)}+(t-t_0)\overbar{B(M)}\Big) \Big)\\
& = \rho_t \Big( \bigcap\limits_{n=1}^\infty\Big(\overbar{ C+B(1/n)}+(t-t_0)\overbar{B(M)}\Big) \Big) = \rho_t\Big(\overbar{C}+(t-t_0)\overbar{B(M)}\Big).
\end{align*}
The proof is finished.
\end{proof}

\begin{prop}\label{lem:asump}
The family of measures $\{\mu_t\}_{t\in[0,T]}$ from Definition \ref{def:MFL} is steadily flowing provided that $\alpha\leq 2$. 
\end{prop}

\noindent
Before we proceed with the proof let us provide a useful lemma, which serves both as a first step and as an interesting piece of information in itself.

\begin{lem}\label{lem:disip}
The family of measures $\{\mu_t\}_{t\in[0,T]}$ from Definition \ref{def:MFL} dissipates kinetic energy a.e., namely the function $t\mapsto E[\mu_t]$ is non-increasing for a.a. $t\in[0,T]$ and we have
\begin{equation}\label{disipineq}
    \int_{t_1}^{t_2} D[\mu_t]\dd t \leq E[\mu_{t_1}] - E[\mu_{t_2}] 
\end{equation} 
for a.a. $t_1\leq t_2\in[0,T]$.
 Without a loss of generality we may assume that  $t\mapsto E[\mu_t]$ is non-increasing and continuous after restriction to $A_E$, defined in Lemma \ref{lem:zbiegze}.
\end{lem}
\begin{proof}
Each $\mu^N$ satisfies Definition \ref{weakkinetic} and Proposition \ref{relaxdef} and thus for all $0\leq t_1\leq t_2\leq T$ we have
$$ E[\mu^N_{t_1}]-E[\mu^N_{t_2}] = \int_{t_1}^{t_2}D[\mu^N_t]\dd t\geq 0. $$
Then, by Definition \ref{def:MFL}, passing with $N\to \infty$ we obtain
$$ E[\mu_{t_1}]-E[\mu_{t_2}] \geq \liminf_{N\to\infty} \int_{t_1}^{t_2}D[\mu^N_t]\dd t\geq \int_{t_1}^{t_2}D[\mu_t]\dd t\geq 0$$
for all $t_1\leq t_2\in A_E$, which marks the end of the proof. The above inequalities with limes inferior and $D[\mu_t]$ follow by the Portmanteau theorem from the narrow convergence in Definition \ref{def:MFL} and from the fact that  the singular integrand in 
$$ D[\mu_t] = \int_{\R^{4d}\setminus\Delta}\frac{|v-v'|^2}{|x-x'|^\alpha}\dd[\mu_t\otimes\mu_t]\otimes\dd t$$
is nonnegative and lower-semicontinuous.
\end{proof}

With Lemma \ref{lem:disip} at hand we are ready to proceed with the proof of Proposition \ref{lem:asump}.

\begin{proof}[Proof of Proposition \ref{lem:asump}]
According to condition (SF), our goal is to establish existence of a full measure set $A\subset[0,T]$ such that for any $t_0\in A$ and any $0\leq \phi\in C_c^\infty(\R^d)$ the family of Radon measures $\{\rho_{t_0,t}[\phi]\}_{t\in A}$ is narrowly continuous i.e. for any sequence $A\ni t_n\to t_0\in A$ we have
\begin{equation}\label{koniec5}
    \int_{\R^d}\xi(x)\dd \rho_{t_0,t_n}[\phi](x)\longrightarrow \int_{\R^d}\xi(x)\dd \rho_{t_0,t_0}[\phi](x)
\end{equation}
for all bounded and continuous functions $\xi:\R^d\to \R$.

The idea of the proof is simple. We test \eqref{weakeqkinetic} using the cut-off function $\psi_{t_1,t_2,\epsilon}$ from \eqref{cutof} multiplied by functions depending on $x$ and $v$, and we upper-bound the right-hand side using Lemma \ref{lem:disip}. Passing with $\epsilon\to 0$ we use the Lebesgue-Besicovitch differentiation theorem to extract full-measure subsets out of the continuity set $A_E$ (see $A_E$ in Lemma \ref{lem:disip}). It has to be done countably many times so that we cover a dense subset of the required test functions. The main problem is that in order to obtain (SF) we need to ensure that one of the end points in $\psi_{t_1,t_2,\epsilon}$, say $t_1$, can be equal to $t_0$. Since the way we extract the domain for $t_1$ and $t_2$ depends on $t_0$ this is far from obvious. What finally solves the problem is the fact that the function
$$ (t_0,t)\mapsto T^{t_0,t}_\#\mu_t $$
is narrowly continuous with respect to $t_0$ for any fixed $t$ (even though it may be discontinuous as a function of $t$ for a fixed $t_0$). The idea is simple, yet its execution turns out to be quite tedious. It follows in four steps.


\medskip
\noindent
$\diamond$ {\sc Step 1.} {\it The case of fixed smooth $\phi$ and $\xi$ and a fixed $t_0$.}

\smallskip
\noindent
 We aim to prove the following claim.

{\it \underline{Claim A}. Fix $\phi,\xi\in C_c^\infty(\R^d)$. Let $D_E$ be a dense countable subset of $A_E\subset (0,T)$ from Lemma \ref{lem:disip} and fix any $t_0\in D_E$. Then for any $t_1\in A_E$ and any $\eta>0$ there exists $\delta>0$ (depending  on $t_1$, $\phi$ and $\xi$), such that for all $t_2\in A_E$ with $|t_1-t_2|<\delta$ we have
\begin{equation*}
    \left|\int_0^T\int_{\R^{2d}}\partial_t\psi_{t_1,t_2,\epsilon}(t)\xi(x)\phi(v)\dd T^{t_0,t}_\#\mu_t \dd t\right|\leq \eta,
\end{equation*}
where $T^{t_0,t}$ was introduced with condition (SF) and $\psi_{t_1,t_2,\epsilon}$ is given in \eqref{cutof}.
}

\smallskip
\noindent
 To prove the above claim first fix $\phi$, $\xi$ and $t_0$ as required and
let
\begin{equation}\label{defL}
    L_{\phi,\xi}:=\max\{|\phi|_\infty, [\phi]_{Lip}, |\nabla_v\phi|_\infty, [\nabla_v\phi]_{Lip}, |\xi|_\infty, [\xi]_{Lip}, |\nabla \xi|_\infty, [\nabla\xi]_{Lip}  \}.
\end{equation} 
 Fix $t_1\in A_E$ and $\eta>0$. Then, by Lemma \ref{lem:disip}, there exists (and we shall fix it) $\delta>0$ such that
\begin{equation}\label{koniec4}
\forall t\in A_E\quad |t-t_1|<\delta\, \Rightarrow\, |E[\mu_t]-E[\mu_{t_1}]|<\frac{\eta}{4}.
\end{equation}
For our fixed $t_1$ fix any $t_2\in A_E$ such that $|t_2-t_1|<\delta$ and consider the test function
$$ \phi_\epsilon(t,x,v):= \psi_\epsilon(t)\xi(\pi_x\left(T^{t_0,t}(x,v)\right))\phi(v), $$
where $\psi_\epsilon(t):=\psi_{t_1,t_2,\epsilon}(t)$ and $\pi_x$ is the projection onto the $x$-coordinate. Clearly, $\phi_\epsilon$ is an admissible test function in the weak formulation \eqref{weakeqkinetic} for the approximative solutions $\mu^N$, which amounts to
\begin{align}\label{lem:asump-1}
    \int_0^T\int_{\R^{2d}}\partial_t\phi_\epsilon + v\cdot\nabla_x\phi_\epsilon \dd \mu_t^N \dd t = \int_{t_1-\epsilon}^{t_2+\epsilon}\int_{\R^{4d}\setminus\Delta}\underbrace{\frac{(\nabla_v\phi_\epsilon-\nabla_v\phi_\epsilon')\cdot(v-v')}{|x-x'|^\alpha}}_{{\mathcal G}:=}\dd (\mu^N_t\otimes\mu^N_t)\dd t,
\end{align}
where we assume without a loss of generality that $t_1<t_2$.
We rewrite and upper-bound the integrand ${\mathcal G}$ on the right-hand side above as
\begin{align*}
|{\mathcal G}|:=\psi_\epsilon(t)\left| \frac{\Big(\nabla_v\big(\xi(x-(t-t_0)v)\phi(v)\big)-\nabla_v\big(\xi(x'-(t-t_0)v')\phi(v')\big)\Big)\cdot(v-v')}{|x-x'|^\alpha} \right|\\
\leq\psi_\epsilon(t)\left| \frac{\Big(\nabla\xi(x-(t-t_0)v)(t-t_0)\phi(v)-\nabla\xi(x'-(t-t_0)v')(t-t_0)\phi(v')\Big)\cdot(v-v')}{|x-x'|^\alpha}\right|\\
\qquad\qquad + \psi_\epsilon(t)\left|\frac{\Big(\xi(x-(t-t_0)v)\nabla_v\phi(v)-\xi(x'-(t-t_0)v')\nabla_v\phi(v')\Big)\cdot(v-v')}{|x-x'|^\alpha} \right| =: I + II
\end{align*}
and further, up to constants depending on $T$ and $L_{\phi,\xi}$ but independent of $t_0$, as
\begin{align*}
    I &\lesssim \psi_\epsilon(t)\left| \nabla\xi(x-(t-t_0)v)\frac{(\phi(v)-\phi(v'))\cdot(v-v')}{|x-x'|^\alpha}\right|\\
    &\qquad\qquad + \psi_\epsilon(t)\left|\phi(v')\frac{\Big(\nabla\xi(x-(t-t_0)v)-\nabla\xi(x'-(t-t_0)v')\Big)\cdot(v-v')}{|x-x'|^\alpha}  \right|\\
    &\lesssim \psi_\epsilon(t)\frac{|v-v'|^2}{|x-x'|^\alpha} + \psi_\epsilon(t)\frac{|x-x'||v-v'|}{|x-x'|^\alpha}
\end{align*}
and as
\begin{align*}
II &\leq  \psi_\epsilon(t)\left|\xi(x-(t-t_0)v)\frac{(\nabla_v\phi(v)-\nabla_v\phi(v'))\cdot(v-v')}{|x-x'|^\alpha}\right|\\
&\qquad\qquad + \psi_\epsilon(t)\left|\nabla_v\phi(v')\frac{\big(\xi(x-(t-t_0)v)-\xi(x'-(t-t_0)v')\big)\cdot(v-v')}{|x-x'|^\alpha}\right|\\
    &\lesssim \psi_\epsilon(t)\frac{|v-v'|^2}{|x-x'|^\alpha} + \psi_{\epsilon}(t)\frac{|x-x'||v-v'|}{|x-x'|^\alpha}.
\end{align*}
The above inequalities can be further bounded using Young's inequality
$$ \frac{|x-x'||v-v'|}{|x-x'|^\alpha}\leq \frac{1}{2}\frac{|v-v'|^2}{|x-x'|^\alpha} + \frac{1}{2}|x-x'|^{2-\alpha} $$
to obtain
\begin{equation}\label{koniec3}
    |{\mathcal G}|\leq C_{\phi,\xi}\psi_\epsilon(t) \frac{|v-v'|^2}{|x-x'|^\alpha} + C_{\phi,\xi} \psi_\epsilon(t)|x-x'|^{2-\alpha},
\end{equation}
where $C_{\phi,\xi}$ is a constant depending on $T$ and $L_{\phi,\xi}$ only.
To deal with $\psi_\epsilon$ above, observe that, since $|t_2-t_1|<\delta$, for sufficiently small $\epsilon>0$ we have ${\rm spt}(\psi_\epsilon)\subset [t_1-\epsilon,t_2+\epsilon]\subset (t_1-\delta,t_1+\delta)$. Since $A_E$ is of full measure, there exist $t_1'$ and $t_2'$ in $A_E$ such that $t_1'\leq t_1-\epsilon < t_2+\epsilon\leq t_2'$ yet still $|t_i'-t_1|<\delta$, $\in\{1,2\}$. Thus, noting that $\psi_\epsilon(t)\leq \chi_{[t_1',t_2']}(t)$ for all sufficiently small $\epsilon>0$ we can further bound \eqref{koniec3} from the above by
$$ |{\mathcal G}|\leq  C_{\phi,\xi}\chi_{[t_1',t_2']}(t) \frac{|v-v'|^2}{|x-x'|^\alpha} + C_{\phi,\xi}\chi_{[t_1',t_2']}(t) |x-x'|^{2-\alpha}.$$
Plugging the above upper-bound back into \eqref{lem:asump-1}  we end up with
\begin{align*}
    \left|\int_0^T\int_{\R^{2d}}\partial_t\phi_\epsilon + v\cdot\nabla_x\phi_\epsilon \dd \mu_t^N \dd t\right| &\leq C_{\phi,\xi} \int_{t_1'}^{t_2'}D[\mu^N_t]\dd t
    + C_{\phi,\xi}\int_{t_1'}^{t_2'}\int_{\R^{4d}} |x-x'|^{2-\alpha} \dd(\mu^N_t\otimes\mu^N_t)\dd t\\
    &\leq C_{\phi,\xi}\left|E[\mu^N_{t_1'}]-E[\mu^N_{t_2'}]\right| + C_{\phi,\xi}|2(T+1)M|^{2-\alpha}|t_2'-t_1'|,
\end{align*}
where the last inequality is based on the uniform boundedness of the support of $\mu^N_t$.
 By Definition \ref{def:MFL}, $E[\mu^N_t]\xrightarrow{N\to\infty} E[\mu_t]$ on $A_E\supset \{t_1',t_2'\}$. Therefore passing with $N\to\infty$ above yields

\begin{align*}
    \frac{1}{C_{\phi,\xi}}\left|\int_0^T\int_{\R^{2d}}\partial_t\phi_\epsilon + v\cdot\nabla_x\phi_\epsilon \dd \mu_t \dd t\right| &\leq 
     \left|E[\mu_{t_1'}]-E[\mu_{t_2'}]\right| + |2(T+1)M|^{2-\alpha}|t_2'-t_1'|\\
     &\leq \left|E[\mu_{t_1'}]-E[\mu_{t_1}]\right| + \left|E[\mu_{t_2'}]-E[\mu_{t_1}]\right|\\
     &\qquad\qquad + |2(T+1)M|^{2-\alpha}|t_1'-t_1| + |2(T+1)M|^{2-\alpha}|t_2'-t_1|.
\end{align*}
 Thus recalling \eqref{koniec4} and that $|t_i'-t_1|<\delta$, $i\in\{1,2\}$, and possibly multiplying $\eta$ by a constant depending on $L$ and $T$, we have
\begin{align*}
    \left|\int_0^T\int_{\R^{2d}}\partial_t\phi_\epsilon + v\cdot\nabla_x\phi_\epsilon \dd \mu_t \dd t\right| \leq C_{\phi,\xi}\frac{\eta}{2} + 2C_{\phi,\xi}|2(T+1)M|^{2-\alpha}\delta.
\end{align*}
Finally, regardless whether $\alpha<2$ or $\alpha = 2$, we can make our $\delta$ smaller if needed, and arrive at
\begin{equation}\label{przedostatnie1}
    \left|\int_0^T\int_{\R^{2d}}\partial_t\phi_\epsilon + v\cdot\nabla_x\phi_\epsilon \dd \mu_t \dd t\right| \leq C_{\phi,\xi}\eta.
\end{equation}
To finish the first step let as unravel the left-hand side of the above inequality. Using the pushforward change of variables formula \eqref{changeofvariables} and the chain rule we obtain
\begin{equation}\label{przedostatnie2}
\begin{split}
    \int_0^T\int_{\R^{2d}}\partial_t\phi_\epsilon + v\cdot\nabla_x\phi_\epsilon \dd \mu_t \dd t& = \int_0^T\int_{\R^{2d}}\partial_t\psi_\epsilon(t)\xi(x-(t-t_0)v)\phi(v)\dd \mu_t \dd t \\
    &+ \int_0^T\psi_\epsilon(t)\int_{\R^{2d}}\underbrace{\partial_t(\xi(x-(t-t_0)v))\phi(v)+ v\cdot\nabla\xi(x-(t-t_0)v)\phi(v)}_{=0} \dd \mu_t \dd t \\
    & = \int_0^T\int_{\R^{2d}}\partial_t\psi_\epsilon(t)\xi(x)\phi(v)\dd T^{t_0,t}_\#\mu_t \dd t.
    \end{split}
\end{equation}
Combining \eqref{przedostatnie1} with \eqref{przedostatnie2} we conclude the proof of the claim and the first step is finished. The constant $C_{\phi,\xi}$ in \eqref{przedostatnie1} can be incorporated into $\eta$ as long as we remember that $\delta$ depends on $L_{\phi,\xi}$.

\medskip
\noindent
$\diamond$ {\sc Step 2.} {\it Passing with $\epsilon\to 0$.}

\smallskip
\noindent
By  Definition \ref{def:MFL} our family of measures $\{\mu_t\}_{t\in[0,T]}$, and thus also $\{T^{t_0,t}_\#\mu_t\}_{t\in[0,T]}$ have uniformly compact supports contained in some large compact set $K$. Moreover, by Stone-Weierstrass theorem there exists a countable set ${\mathcal D}\subset C_c^\infty(\R^d)$ which is dense in $C_b(K)$.
 We aim to prove the following claim.

{\it \underline{Claim B}. There exists a full measure set $A\subset[0,T]$ such that for all $t_0\in D_E$ and all $\phi,\xi\in{\mathcal D}$ the following equivalent statements hold true.  

\smallskip
\noindent
For all $t_1\in A$ and all $\eta>0$ there exists $\delta>0$ (depending on $t_1$, $\phi$ and $\xi$) such that for all  $t_2 \in A$ with $|t_2-t_1|<\delta$ we have
\begin{equation*}
    \left|\int_{\R^{2d}}\xi(x)\phi(v)\dd T^{t_0,t_2}_\#\mu_{t_2} - \int_{\R^{2d}}\xi(x)\phi(v)\dd T^{t_0,t_1}_\#\mu_{t_1}\right|\leq\eta.
\end{equation*}

\smallskip
\noindent
In other words for all $t_1\in A$ and all $\{t_n\}_{n=1}^\infty\subset A$ such that $t_n\to t_1$ we have
\begin{equation*}
    \int_{\R^{2d}}\xi(x)\phi(v)\dd T^{t_0,t_n}_\#\mu_{t_n} \to \int_{\R^{2d}}\xi(x)\phi(v)\dd T^{t_0,t_1}_\#\mu_{t_1}.
\end{equation*}
}

\smallskip
\noindent
To prove the above claim take any $t_0\in D_E$ and any $\phi_1,\xi_1\in{\mathcal D}$. Then Claim A holds meaning that for all $t_1\in A_E$ and all $\eta>0$ there exists $\delta>0$ as in Claim A. We aim to pass with $\epsilon$ to $0$. The argument is similar to that in the proof of Proposition \ref{compactness}, where we already calculated derivatives of $\psi_\epsilon$. Observe that this time around, the function $t\mapsto \int_{\R^{2d}}\xi_1(x)\phi_1(v)\dd T^{t_0,t}_\#\mu_t$ is not necessarily continuous but it is integrable, since both $\phi_1$ and $\xi_1$ are bounded.  Thus we can apply the Lebesgue-Besicovitch differentiation theorem to ensure the existence of a full measure subset $A_{1,t_0}\subset A_E$, depending on the test functions $\phi_1,\xi_1$ and on $t_0$, such that
$$\lim_{\epsilon\to 0} \left|\int_0^T\int_{\R^{2d}}\partial_t\psi_{t_1,t_2,\epsilon}(t)\xi(x)\phi(v)\dd T^{t_0,t}_\#\mu_t \dd t \right|= \left|\int_{\R^{2d}}\xi(x)\phi(v)\dd T^{t_0,t_2}_\#\mu_{t_2} - \int_{\R^{2d}}\xi(x)\phi(v)\dd T^{t_0,t_1}_\#\mu_{t_1}\right|$$
for all $t_1,t_2\in A_{1,t_0}$. Then if additionally $|t_1-t_2|<\delta$ (with $\delta$ depending only on $t_1$ and the norms in $L_{\phi,\xi}$ in \eqref{defL}), we have
$$ \left|\int_{\R^{2d}}\xi(x)\phi(v)\dd T^{t_0,t_2}_\#\mu_{t_2} - \int_{\R^{2d}}\xi(x)\phi(v)\dd T^{t_0,t_1}_\#\mu_{t_1}\right|<\eta. $$

\noindent
This procedure can be repeated indefinitely for all $i=1,2,...$ and all $\phi_i,\xi_i\in {\mathcal D}$ and all  $t_0\in D_E$ leading to the existence of a full measure subset 
$$A:=\bigcap_{i\in{\mathbb N}, t_0\in D_E}A_{i,t_0}\subset A_E$$
satisfying Claim B.

\medskip
\noindent


\medskip
\noindent
$\diamond$ {\sc Step 3.} {\it Taking $t_0 = t_1$.}

\smallskip
\noindent
First let us note the obvious problem that in order to have (SF) we need to make sure that we actually can take $t_0=t_1$ in Claim B. It is unclear since set $A$ may be disjoint with $D_E$ and Claim B  holds only for $t_0\in D_E$. However, the pushforward measure $T^{t_0,t_1}_\#\mu_{t_1}$ is designed so that for any fixed $t_1\in A$ and any $\phi,\xi\in {\mathcal D}$ the function 

$$ t_0\mapsto \int_{\R^{2d}}\xi\phi\dd T^{t_0,t_1}_\#\mu_{t_1} $$
is continuous. To see it take $t_n\to t_0$ for some  $\{t_n\}_{n=1}^\infty\subset D_E$ and $t_0\in A$. Then we have

\begin{equation}\label{ciagpot0}
    \left|\int_{\R^{2d}}\xi\phi\dd T^{t_n,t_1}_\#\mu_{t_1} - \int_{\R^{2d}}\xi\phi\dd T^{t_0,t_1}_\#\mu_{t_1}\right| \leq M[\xi]_{Lip}|\phi|_\infty|t_n-t_0|,
\end{equation}
which not only converges to $0$ but it does so independently of $t_1$ (in fact we have Lipschitz continuity). With that in mind we aim to prove the following claim.

{\it \underline{Claim C}. For all $t_0\in A$ and all $\phi,\xi\in{\mathcal D}$ and any $\{t_n\}_{n=1}^\infty\subset A$, such that $t_n\to t_0$, we have}
$$ \int_{\R^{2d}}\xi\phi\dd T^{t_0,t_n}_\#\mu_{t_n} \to \int_{\R^{2d}}\xi\phi\dd T^{t_0,t_0}_\#\mu_{t_0}.$$

To prove the claim fix $t_0\in A$ and  $\phi,\xi\in{\mathcal D}$ and any $\{t_n\}_{n=1}^\infty\subset A$ such that $t_n\to t_0$. Then for all $\eta>0$ there exists $t_0'\in D_E$ such that $|t_0-t_0'|<\eta$. Then, by Claim B, we have
$$ \int_{\R^{2d}}\xi\phi  \dd T^{t_0',t_n}_\#\mu_{t_n} \to \int_{\R^{2d}}\xi\phi  \dd T^{t_0',t_0}_\#\mu_{t_0}.$$

\noindent
Thus using \eqref{ciagpot0} we infer that
\begin{align*}
    \left|\int_{\R^{2d}}\xi\phi  \dd T^{t_0,t_n}_\#\mu_{t_n} - \int_{\R^{2d}}\xi\phi  \dd T^{t_0,t_0}_\#\mu_{t_0}\right|\leq \left|\int_{\R^{2d}}\xi\phi  \dd T^{t_0,t_n}_\#\mu_{t_n} - \int_{\R^{2d}}\xi\phi  \dd T^{t_0',t_n}_\#\mu_{t_n}\right|\\
    + \left|\int_{\R^{2d}}\xi\phi  \dd T^{t_0',t_n}_\#\mu_{t_n} - \int_{\R^{2d}}\xi\phi  \dd T^{t_0',t_0}_\#\mu_{t_0}\right| + \left|\int_{\R^{2d}}\xi\phi  \dd T^{t_0',t_0}_\#\mu_{t_0} - \int_{\R^{2d}}\xi\phi  \dd T^{t_0,t_0}_\#\mu_{t_0}\right|\\
    \leq \left|\int_{\R^{2d}}\xi\phi  \dd T^{t_0',t_n}_\#\mu_{t_n} - \int_{\R^{2d}}\xi\phi  \dd T^{t_0',t_0}_\#\mu_{t_0}\right|+ 2M[\xi]_{Lip}|\phi|_\infty\eta \xrightarrow{n\to\infty} 2M[\xi]_{Lip}|\phi|_\infty\eta. 
\end{align*}
Due to the arbitrarity of $\eta>0$ Claim C is proved.

\medskip
\noindent
$\diamond$ {\sc Step 4.} {\it Relaxing the restrictions on $\phi$ and $\xi$.}

\smallskip
\noindent
Our goal in the final step is to obtain the result of Step 3 for any $\phi\in C_c^\infty(\R^d)$ and $\xi\in C_b(\R^d)$. Aiming at \eqref{koniec5} (and using notation therein)  let us fix any $\phi\in C_c^\infty(\R^d)$ and $\xi\in C_b(\R^d)$ and let $\{t_n\}_{n\in{\mathbb N}}\subset A$ be such that $t_n\to t_0\in A$. Our goal is to prove that 
\begin{equation}\label{step3}
\int_{\R^d}\xi\dd \rho_{t_0,t_n}[\phi] \to \int_{\R^d}\xi\dd \rho_{t_0,t_0}[\phi].
\end{equation}
By the definition of ${\mathcal D}$ and $K$ in Step 2, for every $\eta>0$ there exist functions $\phi_\eta,\xi_\eta\in{\mathcal D}$ such that
$$ \sup_{x\in K} |\xi(x)-\xi_\eta(x)| + \sup_{v\in K} |\phi(v)-\phi_\eta(v)|<\eta.$$
Thus, the following convergence holds by Claim C
\begin{align*}
    \left|\int_{\R^d}\xi\dd \rho_{t_0,t_n}[\phi] - \int_{\R^d}\xi\dd \rho_{t_0,t_0}[\phi]\right| =  \left|\int_{\R^{2d}}\xi\phi\dd T^{t_0,t_n}_\#\mu_{t_n}  - \int_{\R^{2d}}\xi\phi\dd T^{t_0,t_0}_\#\mu_{t_0}\right|\\
    \leq \left|\int_{\R^{2d}}\xi_\eta\phi_\eta\dd T^{t_0,t_n}_\#\mu_{t_n} - \int_{\R^{2d}}\xi_\eta\phi_\eta\dd T^{t_0,t_0}_\#\mu_{t_0}\right| + \left|\int_{\R^{2d}}(\xi_\eta\phi_\eta - \xi\phi)\dd T^{t_0,t_n}_\#\mu_{t_n}\right| + \left|\int_{\R^{2d}}(\xi_\eta\phi_\eta- \xi\phi)\dd T^{t_0,t_0}_\#\mu_{t_0}\right|\\
    \lesssim \left|\int_{\R^{2d}}\xi_\eta\phi_\eta\dd T^{t_0,t_n}_\#\mu_{t_n} - \int_{\R^{2d}}\xi_\eta\phi_\eta\dd T^{t_0,t_0}_\#\mu_{t_0}\right| + \eta(|\xi|_\infty+|\phi|_\infty)\xrightarrow{n\to\infty} \eta(|\xi|_\infty+|\phi|_\infty)
\end{align*}
and due to the arbitrarity of $\eta>0$ convergence \eqref{step3} holds for any $\phi\in C_c^\infty(\R^d)$ and $\xi\in C_b(\R^d)$. The proof is finished.
\end{proof}

\subsection{Existence of solutions to the Euler-alignment system}

We will now finalise the proof of our last main result, Theorem \ref{main3}.

\begin{proof}[Proof of Theorem \ref{main3}]
Fix a compactly supported measure $\rho_0\in{\mathcal P}(\R^d)$  and $u_0\in L^\infty(\rho_0)$. Let the family $\left\{ \mu_t \right\}_{t\in[0,T]}$ be as in Definition \ref{def:MFL}. As explained in Remark \ref{rem:mu=sigmarho} it admits the disintegration \eqref{mudis}, which we recall here for the reader's convenience
    $$ \mu(t,x,v) = \sigma_{t,x}(v)\otimes \rho_t(x)\otimes \lambda^1(t),\quad \mbox{where}\quad \rho_t(x) = \int_{\R^d_v}\dd\mu_t(x,v),\ u(t,x) = \int_{\R^{d}_v}v\dd\sigma_{t,x}(v). $$
 By Definition \ref{def:MFL} and Remark \ref{rem:mu=sigmarho} we have $\rho_t \in C\left( [0,T], \left( \mathcal{ P}(\R^d), d_{BL} \right) \right)$ and  since $\mu$ is uniformly compactly supported, $u \in L^\infty(\rho)$ (measurability is explained in Section \ref{sec:redisint}) and $\rho_t$ is compactly supported with constants as required in item (i) of Definition \ref{weakeuler}.
 
 To prove that the pair $(\rho, u)$ satisfies (ii) let us first point out that by Lemma \ref{lem:disip} and by Propositions \ref{prop:m-p} and \ref{lem:asump} the family $\{\mu_t\}_{t\in[0,T]}$ satisfies the assumptions of Theorem \ref{main1} and thus -- it is monokinetic i.e.
 \begin{equation}\label{mumonodisint}
      \mu(t,x,v) = \delta_{u(t,x)}(v)\otimes \rho_t(x)\otimes \lambda^1(t)  . 
 \end{equation}
  We plug it into the energy dissipation inequality \eqref{disipineq} from Lemma \ref{lem:disip} obtaining (ii) in Definition \ref{weakeuler}.
 
 It remains to prove that the pair $(\rho, u)$ satisfies (iii), i.e. that it actually satisfies the continuity and momentum equations in the weak sense with our fixed initial data $(\rho_0,u_0)$.
    Let us fix any $\phi=\phi(t,x) \in C^1([0,T]\times\R^{2d})$, compactly supported in $[0,T)$. Testing the weak formulation \eqref{weakeqkinetic} for the approximative solution $\mu^N$  with $\phi$ we obtain
    {\small
    \begin{align*}
        0=&\int_{\R^{2d}} \phi(0,x) \dd \mu^N_0 +\int_0^T \int_{\R^{2d}} (\partial_t \phi + v\cdot \nabla_x \phi) \dd \mu_t^N \dd t = \int_{\R^{2d}} \phi(0,x) \dd \rho^N_0 +\int_0^T \int_{\R^{2d}} (\partial_t \phi + v\cdot \nabla_x \phi) \dd \mu_t^N \dd t.
    \end{align*}
    }    
    By Definition \ref{def:MFL} and \eqref{approxini} passing with $N\to\infty$ and then representing $\mu$ in its disintegrated form \eqref{mumonodisint} yields the continuity equation in \eqref{weakmacro}.
    
    The main part of the proof is dedicated to the proof that the pair $(\rho, u)$ satisfies the momentum equation. We perform the proof componentwise.  
    We test with $v_i\phi(t,x)$, where $v_i\in v=(v_1,...,v_d)$, with $\phi(t,x)$ as before, obtaining
    \begin{align*}
       {\mathcal L}^N &:= \int_{\R^{2d}} \phi(0,x) v_i \dd \mu_0^N + \int_0^T\int_{\R^{2d}} (v_i\partial_t \phi + v_i \left( v\cdot \nabla_x\right) \phi ) \dd \mu^N_t \dd t \\
         &= -\frac{1}{2}\int_0^T \int_{\R^{4d}\setminus\Delta} \frac{(\phi - \phi')(v_i-v_i')}{|x-x'|^\alpha} \dd [\mu_t^N \otimes \mu_t^N] \dd t =: {\mathcal R}^N. 
    \end{align*}
    By \eqref{approxini} we have
    $$ \int_{\R^{2d}} \phi(0,x) v_i \dd \mu_0^N = \frac{1}{N}\sum_{k=1}^N\phi(0,x_{k0}^N)(v_{k0}^N)_i = \int_{\R^{d}}\phi(0,x)\dd(m_0^N)_i\xrightarrow{N\to\infty} \int_{\R^{d}}\phi(0,x) u_{i0}\dd \rho_0, $$
    which together with the convergences in Definition \ref{def:MFL} and the monokinetic form \eqref{mumonodisint} yields
    $$ {\mathcal L}^N\xrightarrow{N\to\infty} \int_{\R^{d}}\phi(0,x) u_{i0}\dd \rho_0 +  \int_0^T\int_{\R^{d}} (u_i\partial_t \phi + u_i \left( u\cdot \nabla_x\right) \phi ) \dd \rho_t \dd t =: {\mathcal L}.$$

    
    \noindent
    Now we deal with ${\mathcal R}^N$. For any $m > 0$ we have
    \begin{align*}
        {\mathcal R}^N &= \frac{1}{2}\int_0^T \int_{\R^{4d}\setminus\Delta} \frac{(\phi - \phi')(v_i-v_i')}{\max\{|x-x'|,m\}^\alpha} \dd [\mu_t^N \otimes \mu_t^N] \dd t \\
            &+\frac{1}{2}\Bigg(\int_0^T \int_{\left(\R^{4d}\setminus\Delta\right)\cap \{ |x-x'|< m \}}  \frac{(\phi - \phi')(v_i-v_i')}{|x-x'|^\alpha} \dd [\mu_t^N \otimes \mu_t^N] \dd t \\
            &- \int_0^T \int_{\left(\R^{4d}\setminus\Delta\right)\cap \{ |x-x'|< m \}}  \frac{(\phi - \phi')(v_i-v_i')}{m^\alpha} \dd [\mu_t^N \otimes \mu_t^N] \dd t\Bigg) \\
            &=: {\mathcal G}^N(m)+{\mathcal B}^N(m).
    \end{align*}
For a fixed $m>0$, the integrand in the ''good'' terms ${\mathcal G}^N(m)$ is continuous and equal to $0$ on the diagonal $\Delta$, and thus it is de facto defined on the whole $\R^{4d}$. Consequently, by convergence in Definition \ref{def:MFL}, we have
$$ {\mathcal G}^N(m)\xrightarrow{N\to\infty} \frac{1}{2}\int_0^T \int_{\R^{4d}\setminus\Delta} \frac{(\phi - \phi')(v_i-v_i')}{\max\{|x-x'|,m\}^\alpha} \dd [\mu_t \otimes \mu_t] \dd t =:{\mathcal G}(m).$$
Moreover, since ${\mathcal B^N}(m) = {\mathcal L}^N - {\mathcal G}^N(m)$, the ''bad'' sequence ${\mathcal B}^N(m)$  also converges and
\begin{equation}\label{Gm}
    {\mathcal L}-{\mathcal G}(m) = \lim_{N\to\infty}{\mathcal B^N}(m).
\end{equation} 
Let us upper-bound the right-hand side above. Denote
$$ \{0<|x-x'|<m\}:= \left(\R^{4d}\setminus\Delta\right)\cap \{ |x-x'|< m \}. $$
Making use of the fact that $\phi$ is Lipschitz continuous in $x$ and independent of $v$ we infer that
{\small
    \begin{align}\label{bn}
        \left|{\mathcal B}^N(m)\right| &\le \int_0^T \int_{\{0<|x-x'|<m\}}  \frac{|\phi - \phi'| |v_i-v_i'|}{|x-x'|^\alpha} \dd [\mu_t^N \otimes \mu_t^N] \dd t
         \lesssim \int_0^T \int_{\{0<|x-x'|<m\}} \frac{|v-v'|}{|x-x'|^{\alpha-1}}\dd [\mu_t^N \otimes \mu_t^N] \dd t \nonumber\\
        &\leq\left( \int_0^T \int_{\{0<|x-x'|<m\}} \frac{|v-v'|^2}{|x-x'|^\alpha} \dd [\mu_t^N \otimes \mu_t^N] \dd t \right)^\frac{1}{2}
        \left( \int_0^T \int_{\{0<|x-x'|<m\}} |x-x'|^{2-\alpha} \dd [\mu_t^N \otimes \mu_t^N] \dd t \right)^\frac{1}{2} \\
        & \le \sqrt{E[\mu_0^N]} \left( \int_0^T \int_{\{0<|x-x'|<m\}} |x-x'|^{2-\alpha} \dd [\mu_t^N \otimes \mu_t^N] \dd t \right)^\frac{1}{2}\nonumber
    \end{align}
    }
    Here we have used the H\" older inequality and Proposition \ref{relaxdef}. Now the proof follows divergent paths depending whether $\alpha<2$ or $\alpha=2$ in items (A) and (B) in Theorem \ref{main3}.
    
    $\diamond$ {\it Proof in the case $\alpha<2$.}
    If $\alpha<2$ then following \eqref{bn} we have
    $$ \lim_{N\to\infty}\left|{\mathcal B}^N(m)\right| \lesssim \sqrt{Tm^{2-\alpha}}\xrightarrow{m\to 0} 0,$$
    which, together with \eqref{Gm}, implies that
    $$ {\mathcal L} = \lim_{m\to 0} {\mathcal G}(m). $$
    By the dominated convergence theorem with dominating functions of the form $\frac{|v-v'|^2}{|x-x'|^\alpha}+ |x-x'|^{2-\alpha}$ the limit above is equal to 
    $$ \int_{\R^{d}}\phi(0,x) u_{i0}\dd \rho_0 +  \int_0^T\int_{\R^{d}} (u_i\partial_t \phi + u_i \left( u\cdot \nabla_x\right) \phi ) \dd \rho_t \dd t = -\frac{1}{2}\int_0^T \int_{\R^{4d}\setminus\Delta} \frac{(\phi - \phi')(v_i-v_i')}{|x-x'|^\alpha} \dd [\mu_t \otimes \mu_t] \dd t.  $$

\noindent    
Plugging \eqref{mumonodisint} into the above equation yields
  $$ \int_{\R^{d}}\phi(0,x) u_{i0}\dd \rho_0 +  \int_0^T\int_{\R^{d}} (u_i\partial_t \phi + u_i \left( u\cdot \nabla_x\right) \phi ) \dd \rho_t \dd t = -\frac{1}{2}\int_0^T \int_{\R^{4d}\setminus\Delta} \frac{(\phi - \phi')(u_i-u_i')}{|x-x'|^\alpha} \dd [\rho_t \otimes \rho_t] \dd t.  $$
  Altogether \eqref{weakmacro} is satisfied in the case $\alpha<2$ and the proof of (A) is finished.

 $\diamond$ {\it Proof in the case $\alpha=2$.}
 The proof in the case $\alpha=2$ differs only in the way we upper-bound ${\mathcal B}^N(m)$ in \eqref{bn}. This time we only have
 $$ \lim_{N\to\infty}|{\mathcal B}^N(m)|^2\lesssim \limsup_{N\to\infty} \int_0^T[\mu^N_t\otimes\mu^N_t](\{0\leq|x-x'|\leq m\})\dd t\leq \int_0^T[\mu_t\otimes\mu_t](\{0\leq|x-x'|\leq m\}) \dd t,$$
 which follows by Portmanteau theorem due to the narrow convergence in Definition \ref{def:MFL} and closedness of $[0,T]\times\{0\leq|x-x'|\leq m\}.$ Passing with $m\to 0$ yields
 $$ \lim_{m\to 0}\lim_{N\to\infty}|{\mathcal B}^N(m)|\leq \left(\int_0^T[\mu_t\otimes\mu_t](\Delta) \dd t\right)^\frac{1}{2} = \left(\int_0^T\sum_{k=1}^\infty (\rho_n(t))^2\right)^\frac{1}{2}, $$
 where $\rho_n(t)$ is the mass of $n$th atom of $\rho$ at the time $t$. In particular, if $\rho$ is non-atomic for a.a. $t$ (as we assume in (B) in Theorem \ref{main3}), then the right-hand side above equals $0$ and the proof proceeds as in the case of $\alpha<2$. Thus the proof of (B) is finished.

\end{proof}
\appendix
\section{}
In the appendix we include useful tools and lemmas utilised in the paper as well as a handful of more tedious proofs. We begin by stating the \textit{disintegration theorem}, see \cite{AGS-08}. 
\begin{theo}[\textbf{Disintegration theorem}]\label{disintegration}
For $d_1, d_2\in\mathbb{N}$ denote projection onto the second componenent as $\pi_2:\R^{d_1}\times \R^{d_2} \longrightarrow \R^{d_2}$. For $\mu \in \mathcal{P} \left( \R^{d_1} \times \R^{d_2}\right)$, define its projection onto second factor as $\nu :=(\pi_2)_{\#} \mu \in \mathcal{P}(\R^{d_2})$. Then there exists a family of probabilistic measures $\left\{ \mu_{x_2}\right\}_{x_2 \in \R^{d_2}}\subset \mathcal{P}(\R^{d_1})$, defined uniquely $\nu$ almost everywhere, such that 
\begin{enumerate}[label=(\roman*)] 
    \item The map
    $$
    \R^{d_2} \ni x_2 \longmapsto \mu_{x_2}(B)
    $$
    is Borel-measurable for each Borel set $B \subset \R^{d_1}$;
    \item The following formula 
    $$
    \int_{\R^{d_1}\times\R^{d_2}} \phi(x_1, x_2) \dd \mu(x_1,x_2) = \int_{\R^{d_2}}\left( \int_{\R^{d_1}} \phi(x_1, x_2) \dd \mu_{x_2}(x_1) \right) \dd \nu (x_2)
    $$
    holds for every Borel-measurable map $\phi:\R^{d_1} \times \R^{d_2} \longmapsto [0, \infty)$.
\end{enumerate}
\end{theo}
The family $\{\mu_{x_2} \} $ is called a \textit{disintegration} of $\mu$ with respect to marginal distribution $\nu$. For the sake of simplicity, we shall often refer to formula \textit{(ii)} above as 
\begin{equation}\label{disint_form}
   \mu(x_1, x_2) = \mu_{x_2}(x_1) \otimes \nu(x_2). 
\end{equation}

\subsection{Properties of weak solutions}
Our next goal is to prove that any narrowly continuous and uniformly compactly supported weak solution to \eqref{meso} dissipates kinetic energy.
\begin{proof}[Proof of Proposition \ref{relaxdef}]
Consider the cut-off function
\begin{equation}\label{cutoffbis}
    \psi_{t_1, \varepsilon}(t) := 
    \begin{cases}
        0, \quad t_1+\varepsilon \le t;\\
        1, \quad 0\le t \le t_1 - \varepsilon;\\
        \text{linear on } (t_1-\varepsilon, t_1 + \varepsilon).
    \end{cases}
    \end{equation}
Then $\phi(t,x,v) = \psi_{t_0}(t)|v|^2$ is a suitable test function in our weak formulation. Thus we have
\begin{align*}
    - \int_{\R^{2d}} |v|^2 \dd \mu_0 &= \int_0^T \partial_t\psi_{t_0,\epsilon}(t) \int_{\R^{2d}}|v|^2\dd\mu_t  \dd t 
            +  \int_0^T \psi_{t_0,\epsilon}(t) \int_{\R^{4d}\setminus\Delta} \frac{|v-v'|^2}{|x-x'|^{\alpha}} \dd [\mu_t \otimes \mu_t] \dd t.
\end{align*}
Our solution is continuous in time, when tested with bounded and Lipschitz continuous functions such as $|v|^2$ (restricted to $B(M)$). Thus the function $t\mapsto \int_{\R^{2d}}|v|^2\dd\mu_t = E[\mu_t]$ is continuous. Therefore proceeding as in the proof of Proposition \ref{compactness} we pass with $\epsilon$ to 0 and, by the fundamental theorem of calculus, we obtain
\begin{align*}
    - E[\mu_0] &= E[\mu_{t_0}] 
            +  \int_0^{t_0} D[\mu_t] \dd t.
\end{align*}
Thus the energy equality \eqref{endisineq} is proved for any $t_0\in(0,T)$, which we promptly extend to $t_0\in[0,T]$ by continuity of $t\mapsto E[\mu_t]$.

To prove that each term in \eqref{weakeqkinetic} is well-defined provided that \eqref{endisineq} holds true, we first note that each term in \eqref{weakeqkinetic}, except the last signular one, is well defined by virtue of the fact that $\mu_t$ are probabilistic and -- by item (i) -- compactly supported.
 The last term on the right-hand side is a different story. Using Lipschitz continuity of $\nabla_v\phi$ we infer that
    \begin{align*}
        \frac{\left(\nabla_v \phi(t,x,v) - \nabla_v \phi(t,x',v')\right)\cdot (v - v')}{|x-x'|^{\alpha}} \leq [\nabla_v\phi]_{Lip}\left(|v-v'|^2|x-x'|^{-\alpha}+ |v-v'||x-x'|^{1-\alpha}\right),
    \end{align*}
    where $[\nabla_v\phi]_{Lip}$ denotes the Lipschitz constant of the function $\nabla_v\phi$.
    The above upper-bound is integrable with respect to $[\mu_t\otimes\mu_t]\otimes\lambda^1(t)$. In fact, by the H\" older inequality and (i), we have
    \begin{equation*}
    \begin{split}
        \int_0^T \int_{\R^{4d}\setminus\Delta} \frac{\left(\nabla_v \phi(t,x,v) - \nabla_v \phi(t,x',v')\right)\cdot (v - v')}{|x-x'|^{\alpha}} \dd [\mu_t\otimes \mu_t'] \dd t\\
        \lesssim \int_0^T D[\mu_t]\dd t + \left(\int_0^T D[\mu_t]\dd t\right)^\frac{1}{2}\left(\int_0^T\int_{\R^{4d}\setminus\Delta}|x-x'|^{2-\alpha}\dd [\mu_t\otimes\mu_t']\dd t\right)^\frac{1}{2}\\
        \leq E[\mu_0] + \sqrt{E[\mu_0]} \sqrt{T}((T+1)M)^\frac{2-\alpha}{2},
        \end{split}
    \end{equation*}
    where the last inequality holds only thanks to condition (i), equality \eqref{endisineq} and the fact that $\alpha\leq 2$. 
\end{proof}

We follow with the proof that any weak solution to the kinetic equation \eqref{meso} is locally mass preserving, which we need in the proof of Theorem \ref{main2}.

\begin{prop}[Condition (MP) revisited]\label{solMP}
Weak  solutions to the kinetic equation \eqref{meso} satisfy condition (MP), at least when $C$ is a singleton.
\end{prop}
 \begin{proof}
    Assume without loss of generality that $C=\{0\}$. Fix any $t_0\in[0,T]$ and assume that $\rho_{t_0}(\{0\})>0$, since if $\rho_{t_0}(\{0\})=0$ then there is nothing to prove. By (i) in Definition \ref{weakkinetic}, measures $\mu_t$ are uniformly compactly supported in $(T+1)B(M)\times B(M)$, thus there exists $M'<M$ such that
    $$ {\rm spt}(\mu_t)\subset  (T+1)B(M')\times B(M') \subset\subset (T+1)B(M)\times B(M),\quad \mbox{for all}\quad t\in[0,T].$$
    Let $0\leq \xi\leq 1$ be a smooth function equal to $1$ on $B(M')$ and to $0$ outside $B(M)$. Let $\psi_\epsilon:=\psi_{t_0,t_1,\epsilon}$ be as in \eqref{cutof} (or as in  \eqref{cutoffbis} if $t_0=0$) and let $\eta>0$. Plugging the function 
    $$\phi(t,x,v) = \psi_\epsilon(t)\xi\left(\frac{|x|}{(t-t_0)+\eta}\right)$$
     into \eqref{weakeqkinetic} yields
     \begin{align*}
         \int_0^T\partial_t\psi_\epsilon(t) \int_{\R^{2d}} \xi\left(\frac{|x|}{t - t_0+\eta}\right)\dd\mu_t\dd t\\
         = \int_0^T\psi_\epsilon(t) \int_{\R^{2d}} \xi'\left(\frac{|x|}{t-t_0+\eta}\right)\left(\frac{|x|}{((t-t_0+\eta)^2} - \frac{x\cdot v}{|x|(t-t_0+\eta)}\right)\dd\mu_t\dd t\geq 0,
     \end{align*}
     where the last inequality above follows from the fact that
     $$ \frac{|x|}{t-t_0+\eta}\geq M' \geq |v|\geq \frac{x\cdot v}{|x|}\quad
    \mbox{whenever}\quad 
    \xi'\left(\frac{|x|}{t-t_0+\eta}\right)\neq 0.$$
    Arguing by the fundamental theorem of calculus as in the proofs of Propositions \ref{relaxdef} and \ref{compactness}, we end up with
    $$ \int_{\R^{d}} \xi\left(\frac{|x|}{t_1-t_0+\eta}\right)\dd\rho_{t_1}\geq \int_{\R^{d}} \xi\left(\frac{|x|}{\eta}\right)\dd\rho_{t_0}, $$
    which implies that
    $$ \rho_{t_1}((t_1-t_0 +\eta)B(M))\geq \rho_{t_0}(\eta B(M')). $$
    Families of balls in the above inequality are decreasing with respect to $\eta\to 0$ and converge (in the sense of intersections) to $(t_1-t_0)\overbar{B(M)}$ and $\{0\}$, respectively. Thus intersecting over $\eta>0$ we recover condition (MP). 
\end{proof}

\subsection{Approximation of the initial data to the macroscopic system}

We begin by stating a basic result on approximation of compactly supported probability measures by empirical measures. The following proposition essentially amounts to the law of large numbers.

\begin{prop}\label{blackbox}
Suppose that a bounded, compactly supported $\mu \in \mathcal{M}_+ (\R^d)$ is given. Then there exists a countable set of pairwise distinct points $\left\{\{x_i^N\}_{i=1}^N \right\}_{N=1}^\infty$ such that 
$$
d_{BL}\left( \frac{1}{N} \sum\limits_{i=1}^N \delta_{x^N_i}(x), \mu \right) \xrightarrow{N\rightarrow \infty} 0.
$$
\end{prop}

\noindent
Using the above proposition we aim to approximate any initial data admissible in Theorem \ref{main3} by a sequence of empirical measures. The main issue is to properly define the approximation of the velocity $u_0$.

\begin{lem}\label{approx}
Let $\rho_0 \in \mathcal{P}(\R^d)$ be compactly supported and let $u_0:\R^d_x \mapsto \R^d$ belong to $L^\infty(\rho_0)$. 
Then there exists a sequence $\{\{(x_{i0}^N,v_{i0}^N)\}_{i=1}^N\}_{N=1}^\infty\subset \R^{2dN}$
such that the convergences in \eqref{approxini} hold true.
\end{lem}

\begin{proof}
    Suppose that $\text{spt}(\rho_0) \subset\subset B(M)$ and that  $\|u_0\|_{L^\infty(\rho_0)}\le M$. For a fixed $N\in \mathbb{N}$ define, using Lusin's theorem, a compact $K_N\subset B(M)$ such that
    \begin{equation}\label{KN}
        u_0 \big|_{K_N} \text{ is continuous and }\, \rho_0\left(B(M)\setminus K_N \right) \le \frac{1}{N}.
    \end{equation}
    
    \noindent
    Moreover let
    $$
    \rho_0 = \rho_0 \mres{K_N} + \rho_0 \mres{\left(B(M) \setminus K_N\right)},
    $$
    where $\rho_0 \mres A$ denotes a restriction of $\rho_0$ to $A$, i.e. $\rho_0 \mres A(B) = \rho_0 (A\cap B)$. Each $\rho_0 \mres K_N$ satisfies the assumptions of Propositon \ref{blackbox}, with $(N-1)/N \le ||\rho_0 \mres K_N ||_{TV} \le 1$. Therefore for each $N$ there exists the sequence $\rho_0^{N,k}$ of the form
    \begin{equation}\label{ro0n}
        \rho_0^{N,k} = \frac{1}{k}\sum\limits_{i=1}^k \delta_{x_{i0}^{N,k}}(x) \text{  such that  } d_{BL} \left( \rho_0^{N,k}, \rho_0 \mres K_N \right) \xrightarrow{k \rightarrow \infty} 0.
    \end{equation}
    
 \noindent   
 In particular, for each $N$ there exists $k_1(N)$ large enough so that
    \begin{equation}\label{k1}
         d_{BL} \left( \rho_0^{N,k}, \rho_0 \mres K_N \right) \le \frac{1}{N}\quad \mbox{for all }\ k\geq k_1(N).
    \end{equation}
   
   \noindent
   
    
  
    
\noindent    Eventually we will construct our approximation of $\rho_0$ based on the family $\rho_0^{N,k}$. Meanwhile, our next step is to better understand the second approximation in \eqref{approxini}, which is equivalent to finding measures
    $$ \nu_{j}^N \xrightharpoonup{N\to\infty} (u_{j0} + M)\rho_0,\qquad u_0 = (u_{10},..., u_{d0}), $$
    where convergence above is in the classical sense of narrow topology for positive Radon measures.   
    We decompose it again with respect to $K_N$ as
    $$
    (u_{j0} + M) \rho_0 = (u_{j0}+M) \rho_0 \mres K_N + (u_{j0} + M)\rho_0 \mres \left( B(M) \setminus K_N\right).
    $$
    In general, $u_0$ need not be defined outside of $K_N$, and we are not able to ensure that points $x_i^k$ belong to $K_N$. To overcome this difficulty, for each $N\in \mathbb{N}$ and a set $K_N$ we use Tietze theorem (\cite[Theorem 1, Section 1.2]{EG-15}) to find a continuous extension $\widetilde{u_0}^{K_N}:\R^d \mapsto \R^d$ such that 
    $
    \widetilde{u_0}^{K_N} = u_0 \text{ on } K_N.
    $
    Coming back to the definition of $\rho_0^{N,k}$ let
    $$
    \nu_j^{N,k} := \frac{1}{k} \sum\limits_{i=1}^N (\widetilde{u_{j0}}^{K_N}( x_{i0}^{N,k}) + M) \delta_{x_{i0}^{N,k}}(x).
    $$
    From the convergence in \eqref{ro0n} we infer that for any fixed $N$ and any bounded continuous function $f$, we have
    $$ \int_{B(M)} (\widetilde{u_{j0}}^{K_N}( x) + M)f \dd \rho_0^{N,k} \xrightarrow{k \rightarrow \infty} \int_{B(M)}  (\widetilde{u_{j0}}^{K_N}( x) + M)f \dd \rho_0 \mres K_N  = \int_{B(M)}  ({u_{j0}}( x) + M)f \dd \rho_0 \mres K_N, $$
    where the last equation above follows from the fact that $\widetilde{u_{j0}}^{K_N}=u_{j0}$ for $\rho_0\mres K_N$ a.a. $x$. Therefore
    
    $$
    \nu_j^{N,k}\xrightharpoonup{k\to\infty} ({u_{j0}} +M) \rho_0 \mres K_N\quad \iff\quad
     d_{BL}\left(  \nu_j^{N,k}, ({u_{j0}} +M) \rho_0 \mres K_N \right)\xrightarrow{k \to \infty} 0.
    $$
    
    \noindent
    Therefore, for each $N$ there exists a $k_2(N)$ large enough so that for all $j\in\{1,...,d\}$ we have
    \begin{equation}\label{k2}
         d_{BL}\left(  \nu_j^{N,k}, ({u_{j0}} +M) \rho_0 \mres K_N \right) \le \frac{1}{N}\quad \mbox{for all }\ k\geq k_2(N).
    \end{equation}
      Consequently, both \eqref{k1} and \eqref{k2} are satisfied for all $k\geq k_3(N):=\max\{k_1(N),k_2(N)\}$. Denote
   $$\rho_0^N:=\rho_0^{N,k_3(N)},\quad \nu_{j}^N := \nu_j^{N,k_3(N)}.  $$
    
    \noindent
    We claim that 
    \begin{equation}\label{conv}
        d_{BL}(\rho_0^N, \rho_0) \rightarrow 0\quad \mbox{and}\quad d_{BL}(\nu_j^N,({u_{j0}} +M) \rho_0)\rightarrow 0,\ \mbox{ for all }\ j\in\{1,...,d\}.
    \end{equation}
    Indeed, we have
      $$
    d_{BL}\left( \rho_0^{N}, \rho_0 \right) \le d_{BL} \left( \rho_0^N, \rho_0 \mres K_N\right) + d_{BL} \left( \rho_0\mres K_N, \rho_0 \right) \le \frac{2}{N} \xrightarrow{N \rightarrow \infty} 0,
    $$
    where the last inequality above follows from \eqref{k1} and from the fact that   $
    d_{BL} \left( \rho_0 \mres K_N, \rho_0 \right) \le \frac{1}{N}
    $
    by the definition of the bounded-Lipschirz distance and by \eqref{KN}. Similarly

    \begin{align*}
    d_{BL}(\nu_j^N,({u_{j0}} +M) \rho_0) &\le d_{BL} \left( \nu_j^N, ({u_{j0}}+M)\rho_0 \mres K_N \right) 
    \\ &+ d_{BL}\left( ({u_{j0}}+M)\rho_0 \mres K_N, ({u_{j0}} +M) \rho_0) \right) \le \frac{1+2M}{N}\xrightarrow{N \rightarrow \infty}0
    \end{align*}
and thus both convergences in \eqref{conv} hold true. It remains to shift each $\nu_j^N$ back by $-M$ and finally define our approximation $(\rho_0^N,m_0^N)$ with
$$ m_{0}^N  := \Big(\nu_1^N - M\rho_0^N,...,\nu_d^N-M\rho_0^N\Big)\quad \mbox{and}\quad v_{i0}^N = \Big(\widetilde{u_{10}}^{K_N}(x_{i0}^{N,k_3(N)}),...,\widetilde{u_{d0}}^{K_N}(x_{i0}^{N,k_3(N)})\Big).
    $$

    
\end{proof}


{\footnotesize
\bibliographystyle{abbrv}
\bibliography{main}
}

\end{document}